\documentclass[a4paper]{article}
\usepackage{latexsym,amsmath,amsthm,amssymb,fullpage,bbm}
\usepackage[latin1]{inputenc}
\newcommand{\dR}{\ensuremath{\mathbb{R}}} %reels
\newcommand{\R}{\dR}
\newcommand{\rr}{\mathbb{R}}

\newcommand{\Sp}{\mathbb{S}}
\newcommand{\Pt}{(P_t)_{t\ge 0}}
\newcommand{\1}{\mathbbm{1}}
\newtheorem{theorem}{Theorem}
\newtheorem{proposition}[theorem]{Proposition}
\newtheorem{lemma}[theorem]{Lemma}
\newtheorem{corollary}[theorem]{Corollary}
 \newtheorem{propdef}[theorem]{Definition-Proposition}
 \newtheorem{fact}[theorem]{Fact}

\theoremstyle{remark} \newtheorem{remark}[theorem]{Remark}

\def\BL{{\bf BL}}

\begin{document}   %----------------------------------------------------

\title{Correlation and Brascamp-Lieb  inequalities \\ for Markov semigroups}
\author{F.~Barthe, D.~Cordero-Erausquin, M.~Ledoux and  B.~Maurey}
\date{}
\maketitle

\begin{abstract}
This  paper builds upon several  recent
% references  \cite{C-L-L1}, \cite{C-L-L2}, \cite{B-CE}, \cite{B-CE-M}
works, where semigroup proofs of Brascamp-Lieb %convolution-type 
inequalities are
provided in various settings (Euclidean space, spheres and symmetric groups).
Our  aim is twofold. Firstly, we provide  a general, unifying,    framework based on Markov generators, in order to cover a variety
of examples of interest going beyond previous investigations. Secondly, we put forward the combinatorial reasons for which unexpected 
exponents occur in these inequalities.
\end{abstract}

\tableofcontents
 
\section{Introduction}
A celebrated inequality of Brascamp and Lieb \cite{brasl76bcyi,lieb90gkgm} asserts that given linear surjective maps between Euclidean spaces
$B_i:H\to H_i$, $i=1,\ldots, m,$ and given positive coefficients $(c_i)_{i=1}^m,$ the best constant $C$ such that
for all non-negative measurable functions $f_i:H_i\to \R$ it holds
$$\int_H \prod_{i=1}^m f_i(B_i x)^{c_i} dx \le C \prod_{i=1}^m \left(\int_{H_i} f_i(y) \, dy \right)^{c_i}$$
can be computed by requiring the inequality on centered Gaussian functions only (i.e. of the form  $f_i=e^{-Q_i}$ where $Q_i$
is a positive definite quadratic form). This far-reaching extension of Hölder's inequality found applications
in harmonic analysis but also in convex geometry.
Indeed, a particular case called the geometric Brascamp-Lieb inequality, put forward by Ball \cite{ball89vscr} when $\mathrm{dim}(H_i)=1$,
leads to many precise volume estimates. The general \emph{geometric} version corresponds to the case when for all $i=1,\ldots,m$, 
$B_iB_i^*=\mathrm{Id}_{H_i}$ and $\sum_i c_i B_i^*B_i=\mathrm{Id}_H$, where $B_i^*$ is the adjoint of $B_i$. Under
these hypotheses the optimal constant in the Brascamp-Lieb inequality is $C=1$. More concretely: let $E_1,\ldots, E_m$ 
be vector subspaces of   $\mathbb R^n$ with its canonical Euclidean structure. Denoting by $P_{E_i}$ the orthogonal projection
onto $E_i$, if $\sum_i c_i P_{E_i}=\mathrm{Id}_{\R^n}$ then for all measurable functions $f_i:\R_i \to \R^+$ it holds
$$\int_{\R^n} \prod_{i=1}^m f_i(P_{E_i}x)^{c_i}dx \le \prod_{i=1}^m \left(\int_{E_i} f_i\right)^{c_i}.$$ 

There exist by now many different proofs of the Brascamp-Lieb theorem: symmetrization when $\mathrm{dim}(H_i)=1$ \cite{brasl76bcyi},
 study of Gaussian kernels \cite{lieb90gkgm}, optimal transport \cite{bart98rfbl}. Heat flow derivation were presented in the recent works \cite{C-L-L1} for
 $\mathrm{dim}(H_i)=1$ and \cite{B-C-C-T1} in general: the geometric Brascamp-Lieb inequality is established by interpolating between the left and
right hand side of the inequality, thanks to the Heat semigroup. The case when optimal Gaussian functions exist follows from the geometric case by 
a clever change of variables and turns out to be generic (the non-trivial remaining cases are in a sense ``boundary'' cases and can
be decomposed into simpler ones). So the geometric case is also essential from a theoretical viewpoint. The Heat flow proofs required
a more precise study of the structure of the problem, since the finiteness of the constant and the existence of Gaussian maximizers     
have to be treated beforehand. They  lead to a complete treatment of the equality cases \cite{C-L-L1, B-C-C-T1,V}. They were  also
flexible enough to adapt to other ambient spaces, as observed by  Carlen, Lieb and Loss \cite{C-L-L1} who 
 discovered the following Young type inequality
on the Euclidean sphere $\mathbb S^{n-1}$: for all measurable functions $f_i:[-1,1]\to \R^+$, it holds
\begin{equation}\label{eq:BLsphere} \int_{\mathbb S^{n-1}} \prod_{i=1}^{n} f_i(x_i) \, d\sigma(x) \le
 \prod_{i=1}^{m}\left( \int_{\mathbb S^{n-1}} f_i(x_i)^2 d\sigma(x)\right)^{\frac{1}{2}},
\end{equation}
where $\sigma$ is the uniform probability measure on $\mathbb S^{n-1}$.
This inequality can be understood as a correlation inequality: the coordinates of a uniform random vector on the sphere
are not independent, so there is no Fubini equality. Instead, 
 inequality \eqref{eq:BLsphere} holds and is a lot better than
Hölder's inequality, which would involve $L^n$-norms of the functions. In a sense, the exponent $2$, which turns out to 
be optimal, shows that the coordinate functions are not too far from being independent. The above inequality was extended
to a spherical version of the geometric Brascamp-Lieb inequality in \cite{B-CE-M}.
Carlen, Lieb and Loss also proved a similar inequality for the set of permutations of a finite set and coordinate functions
 \cite{C-L-L2}.

In this paper we provide a general framework based on Markov generators which allows us to unify the existing results, derive extensions and 
clarify the conditions which are required to prove correlation inequalities. 
Decompositions of the identity  as \eqref{eq:decompx} play an important role. In the case of 
functions depending of blocks of coordinates, we put forward a general set of conditions,
which is similar  to the hypotheses of Finner's theorem for product probability spaces \cite{F}, but applies to
particular non-product spaces. See e.g. Propositions~\ref{prop:coordSOn}, \ref{prop:coordSn} and Section~\ref{sssec:finner}.

The structure of the exposition is as follows.
The abstract framework is described 
in Section~\ref{sec:abs} where a general condition is stated. The next sections provide concrete illustrations of  Proposition~\ref{prop:abstract}.
Section~\ref{sec:diff}
deals with the case where our Markov generator is a diffusion, as it is the case in some classical geometric and probabilistic situations. In particular, we shall put forward the algebraic content of our condition in the case of Riemannian Lie groups (with emphasis on the orthogonal group $SO(n)$) and their quotients.  We study
discrete models and their combinatorics in Section~\ref{sec:disc},  and the case where the generator is a sum of squares in Section~\ref{sec:square}. The final section is devoted to related -- dual, more precisely -- entropy inequalities for the marginals of a probability distribution. We state there  an abstract superadditive inequality for the associated Fisher information, that leads to a somewhat different
route to Brascamp-Lieb inequalities.

%%%%%%%%%%%%%%%%%%%%%%%%%%%%%%%%%%%%%%%%%%%%%%%%%%%%%%%

\section{The abstract argument: commuting maps and  \BL-condition}\label{sec:abs}
The basic input is a measurable space $E$ and a Markov semigroup
$\Pt $ acting on functions on $E$, with generator
$L $. We do not discuss here the various questions related to the underlying domain of $L$
and its associated carré du champ operator (see below) as well as the classes of functions 
under consideration. When a given inequality on functions is stated, it is always understood relatively 
to the suitable domains of $(P_t)_{t\ge 0}$, $L$ or $\Gamma$. These are clear in all the continuous or 
discrete illustrations in this work. We refer to \cite{bakr94huts} for an introduction and further details in this 
respect and to \cite{DS} for the discrete setting.
  
The general framework of our study is the following. We introduce $m\ge 1$  measurable spaces $E_i$ and maps $T_i : E \to E_i$, $i=1, \ldots , m$.
We assume that, for each $i=1, \ldots, m$, that  the map $T_i$  \emph{commutes} to $P_t$ or $L$ in the sense that for every $g : E_i \to \rr$, $ L (g\circ T_i ) $
factors through $T_i$: 
\begin{equation}\label{eq:commut}
L(g\circ T_i )=\tilde{g}\circ T_i
\end{equation}
for some $\tilde{g}:E_i\to \R$. In other words,  $L$ (or $P_t$)  leaves invariant the algebra of functions on $E$ of the form $g\circ T_i$.
This means that $P_t$ or $L $ may be projected on $E_i$
and there exists   a Markov generator $L _i$ on $E_i$ such that
\begin{equation*}
 L (g\circ T_i ) = (L _i g) \circ T_i . 
\end{equation*}
We denote below by ${(P^i_t)}_{t\geq 0}$ the semigroup with generator $L _i$.
If follows that $P_t(g\circ T_i)=(P^i_t g)\circ T_i$. 

We aim at understanding how the  ``geometry''  or the ``combinatorics" of the $T_i$'s and of appropriate choice of constants $c_i>0$ ensure that 
$${P_t\left(\prod_{i=1}^m f_i ^{c_i} \circ T_i \right)\le \prod_{i=1}^m \big(P_t (f_i \circ T_i)\big)^{c_i}}$$
for all $f_i :E_i\to \R^+$, $i=1, \ldots, m$.
Since $(P_t(F^{1/c}))^{c} \le (P_t(F^{1/d}))^{d}$ for $c\ge d>0$, we would like to pick the largest possible constants $c_i$'s.  Also, for obvious reasons (pick all the $f_i$ but one to be identically $1$), the $c_i$'s  will belong to $(0,1]$ and the inequalities we consider can be rewritten in terms of $L^{p_i}$-norms for $p_i = \frac1{c_i}$.
 
This problem is of course reminiscent of the Brascamp-Lieb convolution inequalities described in the introduction, and it can as well be interpreted as a correlation problem. This correlation problem has  many ramifications, as we shall see.

We will, in this general framework,  be dealing with inequalities which are valid for the measures $P_t(.)(x)$, uniformly on
the point $x$.
The following main equivalence is implicit in \cite{C-L-L1},  \cite{B-CE}.

\begin{propdef} \label{prop:abstract} Let $c_i $ be non-negative reals and $T_i:E\to E_i$ maps commuting with $L$,
 for $i = 1, \ldots , m$. 
We say $\{c_i, T_i\}$ satisfy the \emph{\BL-condition} if: 
For all  functions $F_i : E \to \rr$, $i = 1, \ldots , m,$ of the form $F_i=g_i\circ T_i$, setting
 $H = \sum_{i=1}^m  c_i F_i $, it holds 
\begin{equation}\label{ineq:L}
e^{-H} L (e^H) 
         \leq   \sum_{i=1}^m c_i \, e^{-F_i} L (e^{F_i}) .
\end{equation}
Then, the following are equivalent:
\begin{itemize}
 \item For all non-negative functions $f_i : E_i \to \rr$, $i = 1, \ldots , m$,
and every $t\geq 0$,
\begin{equation}\label{ineq:Pt}
 P_t \bigg ( \prod_{i=1}^m f_i^{c_i} \circ T_i\bigg ) 
             \leq  \prod_{i=1}^m \big (P_t (f_i \circ T_i)\big ) ^{c_i} .
             \end{equation}
             
 \item  The $\{c_i, T_i\}$ satisfy the \BL-condition 
 \end{itemize}
 \end{propdef}

\begin{proof}
 Let $f_i : E_i \to \rr$, $i = 1, \ldots , m$, be bounded positive functions.
Let $t\geq 0$ and consider 
$$ \alpha (s) = P_s \bigg ( \exp \bigg ( \sum_{i=1}^m c_i \log P_{t-s} (f_i  \circ T_i)\bigg) \bigg),
      \quad 0\leq s\leq t. $$
Set $F_i = \log P_{t-s} (f_i \circ T_i)$, $i=1, \ldots , m$, and $H = \sum_{i=1}^m c_iF_i$. Direct calculations give
$$ \alpha '(s) = P_s \bigg ( L (e^H) - e^H \sum_{i=1}^m c_i \, e^{-F_i} L (e^{F_i}) \bigg ).$$
 Next,
by the commutation property~\eqref{eq:commut}, 
$F_i = \log P_{t-s} (f_i \circ T_i) $ is a function of $T_i$ so that, under \eqref{ineq:L},
$\alpha '(s) \leq 0$ and thus $\alpha (0) \geq \alpha (t)$. Hence \eqref{ineq:Pt} follows from \eqref{ineq:L}. The
converse implication is obtained by differentiating \eqref{ineq:Pt} at $t=0$. 
\end{proof}

\begin{remark}
Given maps $T_i : E \to E_i$, $ i = 1, \ldots , m$, one may not always be able to check
the \BL-condition~ \eqref{ineq:L}. It might be necessary to consider  further bijective maps
$R : E \to E$, $R_i:E_i\to E_i$ and to deal with  $\tilde{T}_i=R_i\circ T_i \circ R : E \to E_i$, $ i = 1, \ldots , m,$
(still assumed to commute with $P_t$) instead of $T_i$. This is exemplified by the paper \cite{B-C-C-T1}
where the Gaussian-extremizable cases of the Euclidean Brascamp-Lieb inequality are reduced to 
the geometric Brascamp-Lieb inequality. Actually this change of variables is also implicit in \cite{C-L-L1}
where the functions $f_i$ are evolving according to different semigroups. 
\end{remark}

%\subsection{Inequalities for the invariant measure}

It is usually of more interest to state Brascamp-Lieb type inequalities with respect
to the invariant measure $\mu$ of the semigroup $\Pt $. 
When $\Pt$ is ergodic with invariant   probability measure $\mu $, we may let
$t \to \infty$ in the local inequality \eqref{ineq:Pt}   and get inequalities of the type
\begin{equation}\label{ineq:mu}
 \int  \prod_{i=1}^m f_i^{c_i} \circ T_i \,d \mu  
                \leq  \prod_{i=1}^m  \bigg ( \int f_i \circ T_i\, d\mu  \bigg ) ^{c_i}. 
                \end{equation}
Actually this can be viewed directly by studying $\beta(t)=\int \prod_i P_t(f_i\circ T_i)^{c_i} d\mu$. Indeed with 
the notation of the above proof
%\begin{eqnarray*} 
$$\beta '(t) 
    = \int  e^H \left( \sum_{i=1}^m c_i \, e^{- F_i} L (e^{F_i})\right)   d\mu 
     = - \int \bigg ( L (e^H) - e^H \sum_{i=1}^m c_i \, e^{- F_i} 
                            L (e^{F_i}) \bigg ) d\mu.  $$ 
%\end{eqnarray*}
Hence integrating from $0$ to $\infty$, the \BL-condition~\eqref{ineq:L} yields \eqref{ineq:mu}. 
Note that the condition $\beta'(t)\ge 0$   may be rewritten in terms of the Dirichlet form
$ {\cal E} (f,g) := \int f (- L g) d\mu $ as
$$ \sum_{i=1}^m c_i \, {\cal E} \big ( e^{H-F_i}, e^{F_i} \big ) \leq  0. $$

\begin{remark}\label{rmk:dimension}
If $\Pt$ has an infinite invariant measure $\mu$, more hypotheses are needed to get a meaningful limit to the local 
bounds as $t\to \infty$. Assume that $\Pt$ is of dimension $n$,
and size $\kappa >0$, in the sense that   for every $\mu $-integrable function $f : E \to \rr $, at any point,
$$ \lim_{t\to \infty } t^{n/2} P_t f = \kappa \int \! f d\mu .$$
If   the semigroups ${(P_t^i)}_{t\geq 0}$ have invariant measures $\mu_i$, dimensions $n_i$ and
sizes $\kappa _i$, $i = 1, \ldots , m,$  and if in addition
$ \sum_{i=1}^m c_i n_i = n, $
we may use $P_t(f_i\circ T_i)= P_t^i(f_i)\circ T_i$ and let $t \to \infty$ in \eqref{ineq:Pt} to get
$$ \int  \prod_{i=1}^m f_i^{c_i} \circ T_i \, d \mu  
                \leq  \kappa ^{-1} 
                \prod_{i=1}^m  \bigg (\kappa _i  \int f_i \, d\mu_i  \bigg ) ^{c_i}.  $$
\end{remark}

%%%%%%%%%%%%%%%%%%%%%%%%%%%%%%%%%%%%%%%%

\section{Examples of diffusion semigroups}\label{sec:diff}
This section is devoted to several examples of illustration of the preceding abstract scheme in case the 
generator $L$ satisfies a chain rule formula.
Recall that the {\em carré du champ}  of the generator $L$ is defined on some suitable algebra of functions by 
\begin{equation}\label{eq:carreduchamp}
\Gamma(f,g)=\frac12(L(fg)-fL g-gL f).
\end{equation}
For simplicity one writes $\Gamma(f)$ for $\Gamma(f,f)$. If $L$ is a diffusion generator (i.e. a linear differential
operator of order 2 without constant term), then the chain rule yields $L(e^f)=e^f\big(L f+\Gamma(f)\big)$. 
 So for $H = \sum_{i=1}^m c_iF_i$,
$$ e^{-H} L (e^H) -  \sum_{i=1}^m c_i \, e^{-F_i} L (e^{F_i}) 
     =   \Gamma (H) - \sum_{i=1}^m c_i \Gamma (F_i) . $$
Hence, we have:

\begin{fact}[\BL-condition in the diffusion case]\label{BLdiffusion}
If $L $ is a diffusion operator, then  the \BL-condition~\eqref{ineq:L} is equivalent to saying that for every functions $f_i : E_i \to \rr$, $i = 1, \ldots , m$,
\begin{equation}\label{cond:Gamma}
 \Gamma \bigg (\sum_{i=1}^m c_i \, f_i \circ T_i \bigg ) =
      \sum_{i,j=1}^m c_i c_j\,  \Gamma (f_i \circ T_i, f_j \circ T_j) 
               \leq \sum_{i=1}^m c_i \, \Gamma (f_i \circ T_i). 
           \end{equation}    
 \end{fact}
 
Depending on the structure, this condition may be expressed more intrinsically in terms
of the operators $T_i$. We investigate several instances below.
%For example, if the $T_i$'s are real-valued, \eqref{cond:Gamma} reads as
%$$ \Gamma \bigg ( \sum_{i=1}^m c_i \theta _i T_i \bigg )
%      =  \sum_{i,j=1}^m c_i c_j \theta _i \theta _j\Gamma (T_i, T_j) 
%        \leq \sum_{i=1}^m c_i \, \theta _i^2 \, \Gamma (T_i)   $$
%for all $\theta _i \in \rr$, $i = 1, \ldots , m$.
%\begin{remark} The classical H\"older inequality corresponds to the choice
%of $T_i : E \to E$ identity map and $\sum_{i=1}^m c_i = 1$.
%Note furthermore that in the preceding  diffusion case, the proof of
%Proposition 1 yields the following generalized form of H\"older's inequality
%$$ P_t \bigg ( \varphi \bigg (\sum_{i=1}^m c_i \psi _i(f_i) \bigg ) \bigg )
%     \leq \varphi \bigg ( \sum_{i=1}^m c_i \psi _i(P_t f_i) \bigg ) $$
%where $\varphi '\geq 0$, $\varphi ' \geq \varphi ''$, $\psi _i'' \leq - {\psi _i'}^2$,
%$i=1, \ldots, m$.
%\end{remark}

\subsection{Riemannian manifolds}
Let us assume that $E$ is a Riemannian manifold and that 
$\Gamma(f)=|\nabla f|^2$. 
This is in particular the
case if $P_t$ is the Heat equation on $E$ associated to the Riemannian Laplacian $\Delta$.
We also assume that the maps $T_i$ are
differentiable. 
Then condition \eqref{cond:Gamma} amounts to the fact that for every $x\in E$, and for all smooth functions $f_i$,
\begin{equation}\label{BLRiem1}
 \Big|\sum_{i=1}^{m} c_i \nabla (f_i\circ T_i)(x)\Big|^2 \le \sum_{i=1}^{m} c_i \big|\nabla (f_i\circ T_i)(x)\big|^2.
 \end{equation}
For each $x\in E$, we introduce the subspace of $T_xE$, the tangent space at $x$, 
\begin{equation}\label{defE_i}
\mathcal E_i(x):=\big\{ \nabla (f_i\circ T_i)(x); f_i:E_i\to \R \big\}\subset T_x E .
\end{equation}
This is the orthogonal of the kernel of $DT_i(x)$, so it is orthogonal to the tangent directions of   
the level set $\{y\in E;\; T_i(y)=T_i(x)\}$. We denote by $P_{\mathcal E_i(x)} $ the orthogonal projection on $\mathcal E_i(x)$ in the Euclidean space $T_x E$.
We can reformulate~\eqref{BLRiem1} using the following well known equivalence, which relies on  the fact that a linear map and its adjoint have the same norm:
For $\mathcal E$  a Euclidean space,  $\mathcal E_i$, $i=1,\ldots,m,$  Euclidean subspaces of $\mathcal E$  and
 $c_1,\ldots, c_m>0$ we have:
$$
  \forall v_i\in  \mathcal E_i, \quad  \Big|\sum_{i=1}^{m} c_i \, v_i \Big|^2 \le \sum_{i=1}^{m} c_i \, | v_i|^2
  \quad \Longleftrightarrow \quad  
  \forall v\in  \mathcal E, \quad \displaystyle \sum_{i=1}^{m} c_i\, \big|P_{\mathcal E_i} v \big|^2 \le | v|^2
  $$
writing $P_{\mathcal E_i}$ for the orthogonal projection onto $\mathcal E_i$.  More concisely, denoting the identity map by $\mathrm{Id}_{\mathcal E}$,
 the latter condition rewrites as an inequality between symmetric maps: $\sum_{i=1}^{m} c_i P_{\mathcal E_i} \le \mathrm{Id}_{\mathcal E}.$
 
Therefore, we see that  \BL-condition amounts here to a ``moving decomposition of the identity'' inequality  in all tangent spaces.

\begin{fact}[\BL-condition in the Riemannian case]\label{BLRiemannian}
In the setting described above,  the \BL-condition~\eqref{ineq:L} is equivalent to saying that for all $x\in E$, 
\begin{equation}\label{eq:decompx}
 \sum_{i=1}^{m} c_i P_{\mathcal E_i(x)} \le \mathrm{Id}_{T_x E}.
\end{equation}
 \end{fact}

Next, we present instances of such decompositions in the case of model spaces.
\medskip

{\it Geometric Brascamp-Lieb inequality in Euclidean space}.
In $\rr^n$, let, for $i = 1, \ldots , m$,  $E_i$, be vector subspaces of dimension $n_i\ge 1$ and let $c_i \ge 0$,   such that 
$$ \sum_{i=1}^m c_i P_{E_i} =  {\rm Id}_{\rr^n} $$
We take of course $T_i:\mathbb R^n \to E_i$ such that $T_i (x) = P_{E_i} x$,
$x \in \rr^n$, $i = 1, \ldots , m$.

If $B$ is a linear map, $\nabla (f\circ B)(x)= {^t}\!B\nabla f(Bx)$ and
$\Delta(f\circ B)(x)=\mbox{Tr}\big( {^t} B \mbox{Hess}f(Bx) B\big)$. It is then clear that 
the generator $L=\Delta-x\cdot \nabla$  of the Ornstein-Uhlenbeck semigroup commutes with the $T_i's$.
Also for all $x\in \R^n$, the spaces $\mathcal E_i(x)$ are simply $E_i$. Hence ~\eqref{eq:decompx}
is guaranteed by the decomposition of the identity induced by the $E_i$'s.
Thus, we get a Brascamp-Lieb inequality for the standard Gaussian measure, which is ergodic for the  Ornstein-Uhlenbeck semigroup:
\begin{eqnarray*}
 \int_{\R^n} \prod_{i=1}^m f_i(P_{E_i}x)^{c_i} e^{-|x|^2/2} \frac{dx}{(2\pi)^{n/2}} 
& \le& \prod_{i=1}^{m} \left( \int_{\R^n} f_i(P_{E_i}x) e^{-|x|^2/2} \frac{dx}{(2\pi)^{n/2}} \right)^{c_i}\\
  &=&  \prod_{i=1}^{m} \left( \int_{E_i} f_i(y) e^{-|y|^2/2} \frac{dy}{(2\pi)^{n_i/2}} \right)^{c_i}.
\end{eqnarray*}
Note that the decomposition of identity rewrites as $\sum_{i=1}^{m} c_i|P_{E_i}x|^2=|x|^2$, hence setting
$g_i(y)=f_i(y)\exp(-|y|^2/2)$ and using the condition $n=\sum c_i n_i$ (take traces in the decomposition
of the identity), we obtain the Euclidean inequality
$$ \int_{\R^n} \prod_{i=1}^m g_i(P_{E_i}x)^{c_i} dx
 \le  \prod_{i=1}^{m} \left( \int_{E_i} g_i(y) \, dy\right)^{c_i}.$$
Alternatively we could have used the Heat semigroup (with generator $\Delta$) to get a local inequality and pass to the limit using
the dimension of this semigroup, as explained in the Remark~\ref{rmk:dimension}. 

\medskip     
           
{\it Geometric Brascamp-Lieb inequality on the sphere}.
The first inequality of this type was established by Carlen, Lieb and Loss \cite{C-L-L1} for coordinate functions 
on the sphere. It involves an unexpected exponent 2. A natural extension in the spirit of the latter Euclidean inequality was given
in \cite{B-CE-M}. It reads as follows:
If $x \in \Sp^{n-1} \subset \rr^n$ (the standard $(n-1)$-sphere), set as before
$T_i (x) = P_{E_i}(x) $, $i = 1, \ldots , m$, where $E_i\subset \R^n$ are subspaces for which we have
$$ \sum_{i=1}^{m}c_i P_{E_i}\le \mbox{Id}_{\R^n}.$$
Then, whenever $f_i$ are non-negative measurable functions on the sphere, such that $f_i$ depends 
only on $E_i$ (that is $f_i(x)=g_i(P_{E_i}(x)$),  for the uniform probability measure $\sigma$ on $\Sp^{n-1}$ we have,
$$ \int_{\Sp^{n-1}} \prod_{i=1}^{m} f_i^{c_i/2} d\sigma \le \prod_{i=1}^{m} \left(\int_{\Sp^{n-1}} f_i \, d\sigma\right)^{c_i/2}.$$

 It is easy to see that
the Laplacian on $\Sp^{n-1}$ commutes to the operators $T_i$. The strategy in \cite{B-CE-M} is to derive 
decompositions of the identity in all tangent hyperplanes to the sphere, thus fulfilling Condition~\eqref{eq:decompx}.
Another approach based on analysis on the orthogonal group will be given next.

%, and actually the Laplacian
%composed with each coordinate map $x \to x_i$ is an ultraspherical operator on $(-1,+1)$.
\medskip

{\it Hyperbolic space}. It is natural to ask for an hyperbolic analogue of the previous statement.
Let us explain, in two dimensions, why the method does not give any interesting correlation inequality.
The natural functionals $T_i$ to consider are the Busemann functions (which basically are the coordinate
in the direction of a point at infinity), they commute with the Laplace operator. In the disk model,
choose $b_1,\ldots,b_m$ on the unit circle and let $T_i$ be the corresponding Busemann functions.
At a point $x$ in the disk the directions $\mathcal E_i(x)$ are simply the lines spanned by the gradients
of the $T_i's$ (the tangent to the geodesic passing through $x$ and going to $b_i$).
When $x$ tends to a point at infinity $b$  which is not one of the $b_i'$s, it is clear that the lines
$\mathcal E_i(x)$ become asymptotically parallel to the line $\R b$. Hence if a decomposition of the identity exists
in all tangent planes we get that $\sum c_i\le 1$. But in this case the decomposition \eqref{eq:decompx} is trivial
since $P_{\mathcal E_i(x)}\le \mathrm{Id}$, and the inequality that we get is nothing else than
Hölder's inequality.

%%%%%%%%%%%%%%%%%%%%%%%%%%%%%%%%%%%%%%%%%%%%%%%%%%%%%%%%%%%%%%%%%

\subsection{Riemannian Lie groups}
In the case of Lie groups (and their quotients), the geometric structure required to have Brascamp-Lieb type inequalities
is very clear and elegant.

The algebraic structure of the problem appears clearly  when functions depending only on some variables are seen as functions invariant under the (right) action of subgroups of isometries. For instance, a function $f(x)$ on $\R^n$ is a function of $x_1$ if and only  $f$ is invariant under all translation leaving $e_1=(1, 0, \ldots, 0)$ invariant.  Note also that a function $f(x)$ on the sphere $S^{n-1}\subset \R^n$ is a function of $x_1$ if and only if $f$ is invariant under all rotations leaving $e_1$ invariant. In this section, we shall extensively use this point of view in the case of compact Riemannian Lie group.

Let $G$ be a connected compact Riemannian Lie group with unit element denoted by $e$. Let $\mathcal G= T_e M$ be associated Lie algebra ;  by assumption, $\mathcal G$ is a Euclidean space.
Let $\mu$ be the normalized bi-invariant Haar measure on $G$. Here we will work with the Laplace-Beltrami operator $\Delta$ as  Markov generator, for which we indeed have that
$$\Gamma(f)=|\nabla f |^2 ,$$
as required in the previous section.

Let $G_i$ be a connected Lie  subgroup of $G$, with Lie algebra $\mathcal G_i\subset \mathcal G$.
A function $f:G\to \R$ is said to be $G_i$-right-invariant if 
$$ f(xg)=f(x),\quad \forall g\in G_i,\quad \forall x\in G.$$
Equivalently $f$ is of the form $g\circ T_i$ where $T_i:G\to G/G_i$ is the canonical projection
onto the right-quotient, defined by $T_i(x)=xG_i$.  In other words, using notation~\eqref{defE_i}, we are interested in the case where, for $x\in G$, 
$$\mathcal E_i (x) =\{\nabla f (x)\; ; \ \textrm{ $f:G\to \R$ is $G_i$-right-invariant} \}.$$
If $f$ is $G_i$-right-invariant, then  for all $v\in \mathcal G_i$ and all $t\in \R$,
$$ f\big(x\exp(tv)\big)=f(x),\quad \forall x\in G.$$
If $f$ is differentiable, we get that 
$$ 0=\frac{d}{dt}\Big|_{t=0} f\big(x\exp(tv)\big)= \langle \nabla f(x), d(L_x)_{e}v \rangle,\quad \forall v\in \mathcal G_i,$$
where $L_x:G\to G$ is the  left-multiplication by $x$. Since $L_x$  is an isometry of $G$, its differential at $e$, 
$ d(L_x)_{e}$, is an isometry between the Euclidean spaces $T_{e}G =  \mathcal G$ and $T_x G$.
In particular,  we will exploit the invariance property in the following form
\begin{equation}\label{invLie}
(d(L_x)_{e})^{-1}\nabla f(x)\in \mathcal G_i^\bot,\quad \forall x\in G.
\end{equation}
Roughly speaking, a $G_i$-right-invariant function $f$ ``depends" only on $\mathcal G_i^\perp$ in the sense that the gradient $\nabla f (x)$ is in the direction $\mathcal G_i^\perp$ transported on $T_ x M$: 
$$\mathcal E_i(x) = d(L_x)_{e}\, \mathcal E_i ,$$
setting   $\mathcal E_i:= \mathcal G_i^{\bot}$.
With this formalism, the condition to have a Brascamp-Lieb inequality boils down to the existence of a decomposition of the
identity in the Lie algebra:

\begin{theorem}\label{th:lie}
   Let $G$ be a connected compact Riemannian Lie group. Let $(G_i)_{i=1}^m$ be connected Lie subgroups and let
  $\mathcal E_i:=\mathcal G_i^\bot$be  the orthogonal complements in the Lie algebra $\mathcal G$  of $G$ of their Lie algebras $(\mathcal G_i)_{i=1}^m$. 
 Assume that for given $d_1,\ldots, d_m>0$  the following inequality holds between symmetric linear maps of $\mathcal G$:
  \begin{equation}\label{eq:condLie}
   \sum_{i=1}^d d_i P_{\mathcal E_i}\le \mathrm{Id}_{\mathcal G}.
   \end{equation}
Then  the \BL-condition~\eqref{ineq:L} is satisfied. In particular, if for $i=1,\ldots, m$, 
  $f_i:G\to \R^+$ is $G_i$-right-invariant, it holds
\begin{equation}\label{BLineqLie}
 \int_G \prod_{i=1}^{m} f_i^{d_i}d\mu \le \prod_{i=1}^{m} \left(\int_G f_i \,d \mu\right)^{d_i}.
 \end{equation}
\end{theorem}

\begin{proof}
  We consider the Heat kernel on $G$. The Laplace Beltrami operator commutes with right multiplication by the 
elements of the group so that the commutation relation is verified, in particular $P_t f_i$ is again $G_i$-invariant.
Next let us check condition \eqref{ineq:L} in the form~\eqref{BLRiem1} put forward in the beginning of the Riemannian case. If for  $i\le n$, $h_i$
is a differentiable $G_i$-invariant function then, then, rewriting~\eqref{invLie} as     
$$d(L_{x^{-1}})_e \nabla h_i (x) \in \mathcal E_i.$$
we get $P_{\mathcal E_i}d(L_{x^{-1}})_e \,\nabla h_i (x) = d(L_{x^{-1}})_e\,  \nabla h_i (x)$.
Using the fact that  $d(L_{x^{-1}})_e$ is an isometry between $T_x M$ and  $\mathcal G$ and the decomposition of the identity in $\mathcal G$, we see that
\begin{eqnarray*}
\Big\|\sum_i d_i\, \nabla h_i (x) \Big\|^2 & = & 
\Big\|(dL_{x^{-1}})\Big(\sum_i d_i\, \nabla h_i (x)\Big) \Big\|^2 
= \Big\|\sum_i d_i\, (dL_{x^{-1}})\nabla h_i (x) \Big\|^2 \\
&\le & \sum_i d_i\, \big\|(dL_{x^{-1}})\nabla h_i (x) \big\|^2 
 =  \sum_i d_i\, \big\|\nabla h_i (x) \big\|^2 .
\end{eqnarray*} The result follows. Equivalently, we could have said that the isometry $  dL_{x} $ pushes forward the decomposition~\eqref{eq:condLie} from $\mathcal G = T_{e} G$ to the decomposition~\eqref{eq:decompx} on $T_x G$.
\end{proof}

\subsubsection{Calculations in $SO(n)$}
We consider subgroups related to the natural action of $SO(n)$ on $\R^n$ and study the relationship between decompositions of the identity 
of $\R^n$ and the ones induced on $\mathcal A_n=so(n)$, the set of antisymmetric $n\times n$ matrices which is the Lie algebra of $SO(n)$.
The Euclidean structure on $\mathcal A_n$ is  given by the Hilbert-Schmidt norm and the corresponding scalar product $\langle A,B\rangle =\mathrm{Tr}({^t}AB)= -\mathrm{Tr}(AB)$.

We will consider as before functions on $SO(n)$ which are right-invariant with respect to subgroups. There exists two natural subgroups associated to a subspace $E\subset \R^n$: $\mathrm{Fix}(E)$ and $\mathrm{Stab}(E)$.

\begin{lemma}\label{lem:fix}
   Let $E$ be a vector subspace of $\R^n$. Consider the group
 $$H=\mathrm{Fix}(E):=\{ U\in SO(n);\; U_{|E}=\mathrm{Id}\}$$
 and let $\mathcal H$ be its Lie algebra. We have $\mathcal H = \{A \in \mathcal A_n  ; \ A_{|E} = 0\}$ and 
if $P_\mathcal E:\mathcal A_n\to \mathcal A_n$ denotes the orthogonal projection onto $\mathcal E:=\mathcal H^\bot$, we have that
$$ \| P_{\mathcal E} (A)\|^2= 2\| P_E A\|^2-\|P_EAP_E\|^2,\quad \forall A\in\mathcal A_n.$$
Moreover a function $f:G\to \R$ is $H$-right-invariant means that $f(U)$ is actually a function of $U_{|E}$. 
\end{lemma}

\begin{proof}
The equality $\mathcal H=\{A\in \mathcal A_n;\; A_{|E}=0\}$ is obvious. Let us check that the orthogonal projection of $A\in\mathcal A_n$
onto $\mathcal H$ is $P_{E^\bot}AP_{E^\bot}$. Indeed the latter is clearly antisymmetric and vanishes on vectors of $E$, so it 
belongs to $\mathcal H$. It remains to check the orthogonality condition: if $B\in \mathcal H$,
\begin{eqnarray*}
   -\langle B, A-P_{E^\bot}AP_{E^\bot}\rangle &=& \mathrm{Tr}\Big(B\big(A-P_{E^\bot}AP_{E^\bot} \big)\Big) \\
    &=& \mathrm{Tr}(BA)-\mathrm{Tr}(BP_{E^\bot}AP_{E^\bot}).
\end{eqnarray*}
Since $B$ vanishes on $E$, $B=B(P_E+P_{E^\bot})=BP_{E^\bot}$ and taking adjoints $P_{E^\bot}B=B$. It is then clear that 
$\mathrm{Tr}(BP_{E^\bot}AP_{E^\bot})= \mathrm{Tr}(BA)$. The orthogonality follows.

Since $\mathcal E=\mathcal H^\bot$ and denoting for shortness $P$ instead of $P_E$, and $I$ instead of $\mathrm{Id}_{\R^n}$, we have
$$ P_{\mathcal E}(A)=A-P_{E^\bot}AP_{E^\bot}=A-(I-P)A(I-P)= PA+AP-PAP.$$
Eventually, since $P_\mathcal E$ is a self-adjoint involution
\begin{eqnarray*}
   \|P_{\mathcal E}(A)\|^2&=& \langle A,P_{\mathcal E}A\rangle =-\mathrm{Tr}(A(PA+AP-PAP)) \\
  &=&-2 \mathrm{Tr}(A^2P)+\mathrm{Tr}(APAP)=2\|PA\|^2-\|PAP\|^2.
\end{eqnarray*}
The statement on $H$-right-invariant functions is easy. Such a function can be viewed as a function
on $SO(n)/H\approx SO(n)/SO(E^\bot) $ which can be identified to the Stieffel manifold of orthogonal frames of size
$\mathrm{dim}(E)$ in $\R^n$. More explicitly, $U_1 H=U_2H$ is equivalent to $U_2^{-1}U_1\in H$, that is 
for all $x\in E$, $U_1(x)=U_2(x)$. Hence the restriction of $U$ to $E$ characterizes the class of $U$ in the quotient.
\end{proof}

\begin{lemma}\label{lem:stab}
   Let $E$ be a vector subspace of $\R^n$. Consider the group
 $$H=\mathrm{Stab}(E):=\{ U\in SO(n);\; U(E)\subset E\}$$ and let $\mathcal H$ be its Lie algebra.
  If $P_\mathcal E:\mathcal A_n\to \mathcal A_n$ denotes the orthogonal projection onto $\mathcal H^\bot$, it holds
$$ \| P_{\mathcal E} (A)\|^2= 2\| P_E A\|^2-2\|P_EAP_E\|^2, \quad \forall A\in\mathcal A_n.$$
Moreover a function $f:G\to \R$ is $H$-right-invariant means that $f(U)$ is actually a function of $U(E)$. 
\end{lemma}
\begin{proof}
   The argument is very similar to the one of the previous lemma. 
First note that
\begin{eqnarray*}
H&=&\{U\in SO(n);\; U(E)=E\}=\{U\in SO(n);\; U(E)\subset E\mbox{ and } U(E^\bot)\subset E^\bot\}.
%\\
%&\approx& SO(E)\times SO(E^\bot). Pas vrai en general il pourrait y avoir produit de deux indirects
\end{eqnarray*}
For a $H$ right-invariant function $f$, $f(U)$ depends only on $UH$. Since $U_1H=U_2H$ is equivalent to $U_1(E)=U_2(E)$,
the quantity $f(U)$ depends on $U(E)$. In other words $f$ factors through
the Grassmann manifold of spaces of dimension $\mathrm{dim}(E)$ in $\R^n$.

One easily checks that $\mathcal H=\{A\in \mathcal A_n;\; A(E)\subset E\mbox{ and } A(E^\bot)\subset E^\bot\}$.
The orthogonal projection for $A\in \mathcal A_n$ onto $\mathcal H$ is $P_E AP_E+P_{E^\bot}A P_{E^\bot}$.
Indeed this is clearly an antisymmetric map for which $E$ and $E^\bot$ are stable. Moreover for $B\in \mathcal H$,
it is clear that $B=P_EBP_E+P_{E^\bot}BP_{E^\bot}$. Hence 
$$
   -\langle B, A-P_E AP_E+P_{E^\bot}A P_{E^\bot}\rangle = \mathrm{Tr}(BA)-\mathrm{Tr}(BP_EAP_E)
     -\mathrm{Tr}(BP_{E^\bot}AP_{E^\bot})=0.
$$
Eventually, since $\mathcal E=\mathcal H^\bot$, $P_{\mathcal E}(A)=A-P_EAP_E+P_{E^\bot}AP_{E^\bot}$. So calculating
as in the previous lemma,  we have $P_{\mathcal E}(A)=PA+AP-2PAP$ and
\begin{eqnarray*}
 \|P_{\mathcal E}(A)\|^2&=& \langle A,P_{\mathcal E}A\rangle =-\mathrm{Tr}(A(PA+AP-2PAP)) \\
  &=&-2 \mathrm{Tr}(A^2P)+2\mathrm{Tr}(APAP)=2\|PA\|^2-2\|PAP\|^2.
\end{eqnarray*}
\end{proof}

The connection between decompositions of identity of $\R^n$ and of $\mathcal A_n$ is explained next.
\begin{proposition}\label{genSOn}
   For $i=1,\ldots, m$, let $c_i>0$, $E_i$ be a vector subspace of $\R^n$ and let $G_i$ be either
   $\mathrm{Fix}(E_i)$ or $\mathrm{Stab}(E_i)$. Denote by $\mathcal E_i=\mathcal G_i^\perp$ the orthogonal of $\mathcal G_i$ (the Lie 
   algebra of $G_i$) in $\mathcal A_n$. We have
$$ \sum_{i=1}^{m} c_i P_{E_i}\le \mathrm{Id}_{\R^n} \ \Longrightarrow \ 
 \sum_{i=1}^{m}\frac{ c_i}{2} P_{\mathcal E_i}\le \mathrm{Id}_{\mathcal A_n}.$$
 As a consequence, if $\sum_{i=1}^{m} c_i P_{E_i}\le \mathrm{Id}_{\R^n}$ then inequality~\eqref{BLineqLie} holds on $G=SO(n)$ (equipped with its uniform probability measure $\mu$) whenever each $f_i(U)$ is a function of $U(E_i)$ or of $U_{|E_i}$, $i=1, \ldots, m$.
\end{proposition}
\begin{proof}
 By Lemma~\ref{lem:fix} and Lemma~\ref{lem:stab}, for any $A\in \mathcal A_n$, $\|P_{\mathcal E_i}(A)\|^2\le 2\|P_{E_i}A\|^2.$
Hence 
\begin{eqnarray*}
   \sum_{i=1}^{m}\frac{c_i}{2}\|P_{\mathcal E_i}(A)\|^2 &\le &  \sum_{i=1}^{m}c_i\|P_{ E_i}A\|^2
   = \sum_{i=1}^{m}c_i \mathrm{Tr}({^t}AP_{E_i}A) \\
     &=&  \mathrm{Tr}\Big({^t}A \big(\sum_{i=1}^{m}c_i P_{E_i}\big)A\Big)\le \mathrm{Tr}({^t}AA)=\|A\|^2.  
\end{eqnarray*}
   
\end{proof}

Note that we have not used the full strength of Lemmata~\ref{lem:fix} and
\ref{lem:stab}, since we have discarded the terms $\|P_{E_i}AP_{E_i}\|^2$. However, in the case where the $E_i$'s are one dimensional subspaces of $\R^n$, 
these terms vanishes, since in this particular case we have
$$P_{E_i} A P_{E_i} = 0, $$
So,  if $E_i = \R u_i$ where the $u_i$'s are norm $1$ vectors  satisfying the decomposition  of the identity
\begin{equation}\label{onedimdec}
\sum_{i=1}^m c_i \, u_i \otimes u_i =  \mathrm{Id}_{\R^n}
\end{equation}
where $u_i \otimes u_i= P_{E_i}$, then we have, with the notation of the Proposition, 
 $$\sum_{i=1}^{m}\frac{ c_i}{2} P_{\mathcal E_i}= \mathrm{Id}_{\mathcal A_n}.$$
We do not loose in the passage to the Lie algebra. A particular case of interest is when $m=n$, $c_1=\ldots =c_n=1$ and $(u_1, \ldots, u_n)$ is an orthonormal basis of $\R^n$.

For higher dimensional $E_i$'s, it is possible, in some specific situations, to recombine the terms $\|P_{E_i}AP_{E_i}\|^2$
to recover a multiple of $\|A\|^2$ and to improve the exponents in the correlation inequality.
This is easily seen for coordinate subspaces, i.e. spaces spanned by vectors of the canonical basis $(e_1, \ldots, e_n)$ of $\R^n$ (or of any given orthonormal basis, of course). The following proposition puts forward a typical set of conditions
in order that \BL-condition \eqref{ineq:L} is fulfilled. 
It will appear later in similar 
forms.
  
\begin{proposition}\label{prop:coordSOn} Let $\mathcal I$ be a collection of subsets of  $\{1,\ldots,n\}$. Assume that it is written
as a disjoint union $\mathcal I=\mathcal I_1\cup \mathcal I_2$. 
For each nonempty subset $I\in \mathcal I$,  let  $c_I\ge 0$, $E_I:=\mathrm{span}(e_i;\; i\in I)$
and $f_I:SO(n)\to \R^+$ such that 
  \begin{itemize}
        \item if $I\in \mathcal I_1$ then  for all $U$, $f_I(U)$ only depends  on $U_{|E_I}$, 
         \item if $I\in \mathcal I_2$ then for all $U$, $f_I(U)$ only depends on $U(E_I)$.
\end{itemize}
If for all $1\le i, j\le n$ with $i\neq j$ it holds:
  $$\displaystyle  \sum_{\stackrel{I\in \mathcal I_1}{I\cap\{i,j\}\neq\emptyset}} c_I
    +  \sum_{\stackrel{I\in \mathcal I_2}{ \mathrm{card}(I\cap\{i,j\})=1}} \!\!\!c_I\le 1,$$
  then \BL-condition\eqref{ineq:L} is satisfied and in particular
  $$ \int_{SO(n)} \prod_{I\in\mathcal I}  f_I^{c_I} d\mu \le \prod_{I\in\mathcal I}  \left(\int_{SO(n)} f_I\, d\mu\right)^{c_I}.$$ 
\end{proposition}
\begin{proof} Simply note that for $A=(a_{i,j})_{1\le i,j\le n}\in \mathcal A_n$, $\|P_{E_I}AP_{E_I}\|^2= \sum_{i,j\in I} a_{i,j}^2$ and 
$$
 \|P_{E_I}A\|^2= \mathrm{Tr}({^t}AP_{E_i}A)= \mathrm{Tr}\Big(\sum_{i\in I} Ae_i\otimes Ae_i\Big)=\sum_{i\in I} \|Ae_i\|^2=
  \sum_{i\in I}\sum_{j=1}^n a_{i,j}^2.
  $$ 
Let us set $\lambda_I:=1$ if $I\in\mathcal I_1$,  $\lambda_I:=2$ if $I\in\mathcal I_2$. Using Lemmata~\ref{lem:fix} and \ref{lem:stab}, and the antisymmetry of $A\in \mathcal A_n$, we have
\begin{eqnarray*}
 &&\sum_I c_I \|P_{\mathcal E_I}(A)\|^2 = \sum_I  c_I \left(2\|P_{E_i} A\|^2-\lambda_I \|P_{E_I}AP_{E_I}\|^2 \right)\\
   &=&  \sum_I  c_I \left( 2\sum_{i\in I}\sum_{j=1}^n a_{i,j}^2-\lambda_I  \sum_{i,j\in I} a_{i,j}^2\right)
   = \sum_{i,j=1}^n a_{i,j}^2 \left(2\sum_{I;\; i\in I} c_I-  \sum_{I;\; i,j\in I} \lambda_I c_I \right) \\
   &=& 2 \sum_{1\le i<j\le n} a_{i,j}^2 \left(\sum_{I;\; i\in I} c_I+\sum_{I;\; j\in I} c_I-  \sum_{I;\; i,j\in I} \lambda_I c_I \right).
\end{eqnarray*}
The latter is upper bounded by $\|A\|^2$ as soon as for all $i\neq j$,
$$ \sum_I c_I\left(\1_{i\in I}+\1_{j\in I}-\lambda_I \1_{i,j\in I} \right)\le 1,$$
which is exactly our hypothesis on the coefficients $(c_I)_{I\in\mathcal I}$.
  Hence $\sum_I c_I P_{\mathcal E_I}\le \mathrm{Id}_{\mathcal A_n}$ and  Theorem~\ref{th:lie} yields the claim.
\end{proof}

Let us restate the previous result in the case we are looking to inequalities involving identical $c_i$'s.
\begin{proposition}
Let  $\mathcal I$ be a family of subsets of $\{1,\ldots,n\}$, and consider 
\begin{eqnarray*}   
  p&:=& \max_{1\le i<j\le n}\mathrm{card}\big\{I\in\mathcal I; \; I\cap \{i,j\}\neq \emptyset\big\},\\
  q&:=& \max_{1\le i<j\le n}\mathrm{card}\big\{I\in\mathcal I; \; \mathrm{card}\big(I\cap \{i,j\}\big)=1\big\},
\end{eqnarray*}  
then for all non-negative functions $g_I, h_I$ defined on suitable spaces,
\begin{eqnarray*}
\int \prod_{I\in \mathcal I} g_I(U_{|E_I}) \, d\mu(U)&\le& \prod_{I\in\mathcal I} \left(\int  g_I(U_{|E_I})^p \, d\mu(U)
\right)^{\frac1p},\\
\int \prod_{I\in \mathcal I} h_I(U(E_I)) \, d\mu(U)&\le &\prod_{I\in\mathcal I} \left(\int  h_I(U(E_I))^q \, d\mu(U)
\right)^{\frac1q}.
\end{eqnarray*}
\end{proposition}
Let us put forward two particular cases of application of the previous result:
\begin{itemize}
        \item Blocks of coordinates: if $\mathcal I$ is a non-trivial partition of $\{1,\ldots,n\}$ then each pair $\{i,j\}$ meets at most
         two sets in the family and we get $p=q=2$.
        \item Loomis-Whitney inequality: if $\mathcal I$ is the family of all subsets of $\{1,\ldots,n\}$ of size $k$, then
         any pair meets ${n \choose k}-{n-2 \choose k}$ sets. Hence we have
         $$p= {n \choose k}-{n-2 \choose k}=
          {n-1 \choose k-1}+ {n-2 \choose k-1}.$$ 
          However the number of sets of cardinality $k$ which intersect a given pair
         in exactly one point is ${n \choose k}-{n-2 \choose k}-{n-2 \choose k-2}=2{n-2 \choose k-1}$. So we get a smaller
         exponent 
         $$q=2{n-2 \choose k-1}.$$
\end{itemize}
         It is worth noting that a direct application of Proposition~\ref{genSOn} would have given \emph{worst} estimates (when $k\ge 2$), in both cases. Indeed, if we denote by $P^I$ the projection  
         onto a subspace spanned by $\{e_i, \ i\in I\}$ for $I\subset\{1, \ldots n\}$, we have
         $$\sum_{|I|=k} \frac{n}{k {n \choose k}}P^I = \mathrm{Id}_{\R^n}$$
         and therefore we would get exponent $p$ and $q$ equal to $ 2\frac{k}n {n \choose k}= 2{n-1 \choose k-1}$.

\begin{remark}
On can take advantage of the terms $\|P_EAP_A\|^2$ in more general situations. They have to be rather symmetric
though. Letting $2\le k\le n-1$, one instance is given by the family  of  all the spaces spanned by any $k$  vertices
of a regular simplex in $\R^n$ with center of mass at the origin. 
\end{remark}

\subsubsection{Passing to quotients}
\label{quotients}
So far, we have taken advantage of right-invariances of the functions $f_i$.
Plainly, similar results hold if all the functions are left-invariant instead.  It would be very interesting
to get better inequalities when the functions $f_i$ enjoy left and right invariances together (this 
would encompass functions on $SO(n)$ depending on matrices $U$ only through submatrices). Unfortunately, our 
approach does not give interesting general results in this direction (nothing better than what one gets by applying
first Hölder's inequality in order to get two integrals; each of these integrals is then upper-bounded
by using only one-sided invariance).  In the specific case when the functions have different right-invariances
and a common left invariance, our results can be stated instead on the left-quotient. This is a way to get
inequalities for  homogeneous spaces corresponding to a compact Riemmanian Lie group. 

Let us  illustrate this remark for the sphere: if $E_i$ is a subspace of $\R^n$ and $f_i:\Sp^{n-1}\to \R^+$
is of the form $f_i(x)=g_i(P_{E_i} x)$, we may introduce $F_i:SO(n)\to \R^+$ defined by 
$F_i(U)=g_i(P_{E_i} {^t}Ue_1)=g_i( {^t}(UP_{E_i})e_1).$ Then $F_i$ is $\mathrm{Fix}(E_i)$-right-invariant and also
$\mathrm{Fix}(\R e_1)$-left-invariant. Applying our results on $SO(n)$ and using the fact that the
law of ${^t}Ue_1$ under the Haar probability measure on $SO(n)$ is the uniform distribution on the sphere
recovers the main result of \cite{B-CE-M} which extends inequality~\eqref{eq:BLsphere}: if $\sum_i c_i P_{E_i}\le\mathrm{Id}_{\R^n}$
then 
$$\int_{\Sp^{n-1}} \prod_{i} f_i^{c_i/2} d\sigma \le \prod_i \left(\int_{\Sp^{n-1}} f_i\, d\sigma\right)^{c_i/2}.$$
Moreover, if $f:\Sp^{n-1}\to \R^+$ is of the form $f(x)=g\big(|P_Ex|\big),$ then the function 
 $F:SO(n)\to \R^+$ defined by 
$F(U)=g\big(|P_{E} {^t}Ue_1|\big) $  is $\mathrm{Stab}(E_i)$-right-invariant and  $\mathrm{Fix}(\R e_1)$-left-invariant.
This allows us to transfer all of our $SO(n)$ results to the sphere. 

Actually, a more general route is to note that~\BL-condition, in the form~\eqref{eq:condLie}, passes to quotient. 
\begin{lemma}
Let $E$ be a Riemannian homogeneous space and $G$ a compact Riemannian Lie group of isometries acting transitively on $E$.  Assume we are in the situation of Theorem~\ref{th:lie}. A function $f:E\to \R$ is said $G_i$-invariant if $f(g\cdot x) = f(x)$ for every $x\in E$ and $g\in G_i$. We can consider the associated $\tilde T_i:E\to E/G_i$ or more simply, with the notation~\eqref{defE_i},
$$\mathcal E_i (x) =\{\nabla f (x)\; ; \ \textrm{ $f:E\to \R$ is $G_i$-invariant} \}.$$
If condition~\eqref{eq:condLie} holds on $G$, then the \BL-condition holds in $E$ in the equivalent form~\eqref{eq:decompx}.
\end{lemma}

\begin{proof}
Fix $x\in E$ and let $G_x=\{g\in G \, ;\ g\cdot x = x\}$.   Then, if  we decompose the algebra $\mathcal G = T_{\mathrm{Id}} G$ (equipped with its Euclidean structure) as an orthonormal sum  $\mathcal G = \mathcal G_x \oplus \mathcal G_x^\perp$ where $ \mathcal G_x$ is the Lie algebra associated to $G_x$, we have that $ \mathcal G_x^\perp$ is isometric to $T_x M$ by the  isometry map
$$\pi=\pi_x \,:\, A \longrightarrow \pi(A)=\frac{d}{dt}_{|t=0} \exp(tA)\cdot x .$$
We see that $\pi(\mathcal G_i) \subset \mathcal E_i(x)^\perp$ and therefore $\mathcal E_i(x) \subset \pi(\mathcal E_i) $. Note that $G_x \subset G_i$ and $\mathcal E_i \subset \mathcal G_x^\perp$. Since $P_{ \pi(\mathcal E_i)}= \pi P_{\mathcal E_i} \pi^{-1}$, we get from~\eqref{eq:condLie}  that
$$\sum_{i=1}^m c_i \,  P_{\mathcal E_i(x)} \le \mathrm{Id}_{T_x E}.$$

\end{proof}

\medskip

It is sometimes necessary to work directly on quotients, in particular for quotients of finite measure with a 
cover of infinite measure. We briefly discuss the example of the flat torus $(\R/\mathbb Z)^n$. We consider 
for $i=1,\ldots,m$, rational vectors $u_i\in\mathbb Q^n$.  For each $i$ let $\ell_i$ be the largest common
divisor of the numbers $\langle u_i,e_1\rangle,\ldots,\langle u_i,e_n\rangle$. In order to define the map $x\mapsto \langle x,u_i\rangle$ on the torus, one has to identify $\langle u_i,e_k\rangle$ to $0$ for all $k$. This amounts to 
quotient $\R$ by $\sum_{k=1}^m \langle u_i,e_k\rangle \mathbb Z=\ell_i \mathbb Z$.
Let $T_i:(\R/\mathbb Z)^n \to \R/\ell_i \mathbb Z$ be the map defined by $T_i(x)=\langle x,u_i  \rangle \mod \ell_i$.
One easily checks that the Laplacian commutes with $T_i$ (same calculation as in $\R^n$). Since for every $x$,
$\nabla (f_i\circ T_i)(x)$ is a multiple of $u_i$, if
 $\sum_{i=1}^m c_i u_i\otimes u_i\le \mathrm{Id}_{\R^n }$ it follows
that 
$$\int_{(\R/\mathbb Z)^n} \prod_i f_i(\langle x,u_i\rangle )^{c_i} dx\le \prod_{i=1}^m 
\left(\int_{(\R/\mathbb Z)^n}  f_i(\langle x,u_i\rangle) dx \right)^{c_i}=
\prod_{i=1}^m 
\left(\int_{\R/\ell_i \mathbb Z}  f_i \right)^{c_i}.$$
%\begin{remark} COMMENTAIRES pas indispensables, on peut aussi parler de l'inegalité de Young???
%In this argument we have applied the viewpoint for Riemannian manifolds rather than the one for Lie groups in
%terms of invariances (but it is possible to do it, since being a function of $T_i$ amounts to 
%the invariance by   translations orthogonals to $u_i$). Also note, that unlike in $\R^n$ we do not need an equality but just an inequality in the
%decomposition of the identity, since the homogeneity is not required any more.
%\end{remark}

\subsection{Dirichlet distributions and their relatives}
For $x\in \R^n$, we set $S(x)=x_1+\cdots+x_n$. Let $\alpha\in (0,+\infty)^n$, then by definition the 
Dirichlet  law  $D_{n-1}(\alpha)$ is the distribution of 
$$ \frac{(X_1,\ldots,X_{n-1})}{X_1+\cdots+X_n}$$
where $X_1,\ldots,X_n$ are independent random variables such that for each $i$, $X_i$
is $Gamma(\alpha_i)$ distributed. More precisely it is supported on $T_{n-1}=\{ y\in\R_+^{n-1}; \;
y_1+\cdots+y_{n-1}\le 1\}$ and 
 $$ D_{n-1}(\alpha)(dy)= \frac{\Gamma(S(\alpha))}{\prod_{i\le n} \Gamma(\alpha_i)} 
 \Big(\prod_{i\le n-1}y_i^{\alpha_i-1}\Big) \Big(1-\sum_{i\le n-1} y_i\Big)^{\alpha_n-1} \1_{T_{n-1}}(y) \, dy.$$
 In order to get more symmetric results, we prefer to work with another representation: we consider
  the law $\widetilde D_{n-1}(\alpha)$   of 
$$ \frac{(X_1,\ldots,X_{n})}{X_1+\cdots+X_n}.$$
It is supported on the regular simplex $\Delta_{n-1}=\{ y\in\R_+^{n}; \;
y_1+\cdots+y_{n}=1\}$ and its density with respect to Lebesgue measure on $\Delta_{n-1}$ is proportional
to $y\mapsto \prod_{i\le n} y_i ^{\alpha_i-1}$.
 Recall that some Dirichlet distributions are closely related to uniform spherical measures. Indeed if $G_i$ are independent variables with distribution $\exp(-t^2) dt/\sqrt{\pi}$,
  then the uniform measure on $\Sp^N$ coincides with the law of 
  $$ \frac{(G_1,\ldots,G_N)}{\sqrt{G_1^2+\cdots+G_N^2}}.$$
  Note that $G_i^2$ has distribution $Gamma(1/2)$. Write $N=k_1+\cdots+k_n$. It is then clear that the  image of the uniform probability on $\Sp^{N-1}$
  by the map
  $$ x\mapsto (x_1^2+\cdots+x_{k_1}^2,\; x_{k_1+1}^2+\cdots+x_{k_1+k_2}^2,\;\ldots,\: x_{k_1+\cdots+k_{n-1}+1}^2+\cdots+x_N^2).$$
  is $\widetilde D_{n-1}(k_1/2,\ldots,k_n/2).$ This allows us  to transfer some of our spherical results, but 
only to Dirichlet laws with half integer coefficients. In order to deal with general coefficients the following
direct study is needed.

The measure $D_{n-1}(\alpha)$ is known (see \cite{E-K,M-R}) to be reversible and ergodic for the following Fleming-Viot operator
$$ L_\alpha f=\sum_{i\le n-1} x_i \partial_{i,i}^2f- \sum_{i,j\le n-1} x_ix_j \partial_{i,j}^2f
  +\sum_{i\le n-1} \big( \alpha_i-S(\alpha)x_i\big) \partial_i f.$$
In the symmetric representation associated to $\widetilde D_{n-1}(\alpha)$, it is natural to consider   
 the operator $\widetilde L_\alpha$ defined for smooth functions $f:\R^n \to \R^+$ and for  $x\in \Delta_{n-1}$
by 
$$ \widetilde L_\alpha f(x)=\sum_{i\le n} x_i \partial_{i,i}^2f(x)- \sum_{i,j\le n} x_ix_j \partial_{i,j}^2f(x)
  +\sum_{i\le n} \big( \alpha_i-S(\alpha)x_i\big) \partial_i f(x).$$
It is not hard to check that $ \widetilde L_\alpha f$ only depends on the restriction of $f$ to $\Delta_{n-1}$
(in the intrinsic formulation $\partial_i g$ is to be understood as $Dg\cdot  P_H e_i=Dg\cdot  (e_i-\1/n)$, where
$\1=(1,\ldots,1)\in \R^n$ and $H=\1^\bot$). However it is convenient to be able to apply $ \widetilde L_\alpha f$
to functions $f$ defined on the whole space. For example if we write $f(y)=g(y_1,\ldots,y_{n-1})$, $y\in \Delta_{n-1}$
then it is clear that $ \widetilde L_\alpha f(y)=  L_\alpha g(y_1,\ldots,y_{n-1})$; hence the properties
of $L_\alpha$ will pass to $ \widetilde L_\alpha f$ (in particular $\widetilde D_{n-1}(\alpha)$ is reversible and
ergodic for the semigroup generated by $ \widetilde L_\alpha f$).

The   carré du champ of $ \widetilde L_\alpha $ can be expressed in the following convenient form, for $x\in\Delta_{n-1}$:
\begin{eqnarray*}
\Gamma(f) &=& \sum_{i\le n} x_i (\partial_i f)^2- \sum_{i,j\le n} x_ix_j \partial_i f \partial_j f \\
  &=& \sum_{i\le n} x_i (\partial_i f)^2 - \Big(\sum_{i\le n} x_i\partial_i f\Big)^2\\
  &=& \frac12 \sum_{i\neq j} x_ix_j (\partial_i f-\partial_j f)^2,
\end{eqnarray*}
where we have noted that $\Gamma(f)$ is actually a variance with respect to the probability measure
$\sum x_i \delta_i$. The last formula comes from the representation $\mathrm{Var}(X)=\frac12 E((X-X')^2)$ where
$X'$ is an independent copy of $X$. We are ready to establish

\begin{proposition}\label{prop:Dirichlet} Let $\mathcal I$ be a collection of subsets of  $\{1,\ldots,n\}$. Assume that it is written
as a disjoint union $\mathcal I=\mathcal I_1\cup \mathcal I_2$. 
For each nonempty subset $I\in \mathcal I$,  let  $c_I\ge 0$, 
and $f_I: \Delta_{n-1}\to \R^+$ such that 
  \begin{itemize}
        \item if $I\in \mathcal I_1$ then  for all $x$, $f_I(x)$ only depends  on $(x_k)_{k\in I}$, 
         \item if $I\in \mathcal I_2$ then for all $x$, $f_I(x)$ only depends on $\sum_{k\in I} x_k$.
\end{itemize}
If for all $1\le i, j\le n$ with $i\neq j$ it holds:
  $$\displaystyle  \sum_{\stackrel{I\in \mathcal I_1}{I\cap\{i,j\}\neq\emptyset}} c_I
    +  \sum_{\stackrel{I\in \mathcal I_2}{ \mathrm{card}(I\cap\{i,j\})=1}} \!\!\!c_I\le 1,$$
  then the \BL-condition~\eqref{ineq:L} is satisfied and if $X$ is $\widetilde{D}_{n-1}(\alpha)$ distributed
  $$ E \Big(\prod_{I\in\mathcal I}  f_I^{c_I}(X)\Big) \le \prod_{I\in\mathcal I} \Big( E f_I(X)\Big)^{c_I}.$$ 
\end{proposition}
\begin{proof}
   First, we check the commutation relations. Since the coordinates play symmetric roles, we may assume
that $I=\{1,\ldots,k\}$. Also we may extend our functions to $\R_+^n$. 
If for all $x$, $g(x)=f(x_1+\cdots+x_n)$ it is obvious that 
$$ \partial_ig(x)= \left\{ 
  \begin{array}{cl}
     0 & \mbox{if } i>k\\
     f'(x_1+\cdots+x_k)  & \mbox{if } i\le k
   \end{array}
 \right. \quad \mbox{ and } \quad
 \partial^2_{i,j} g(x)= \left\{ 
  \begin{array}{cl}
     0 & \mbox{if } i \mbox{ or } j>k\\
     f''(x_1+\cdots+x_k)  & \mbox{if } i, j\le k.
   \end{array}
 \right.
$$
It is then clear that $\widetilde L_\alpha g(x)$ is a function of $x_1+\cdots+x_k$.
Similarly, if $g(x)=h(x_1,\ldots,x_k)$ then $\widetilde L_\alpha g(x)$ is a function of $(x_i)_{i\le k}$.

Next, we have to check the analogue of Condition~\eqref{cond:Gamma}, namely
$$ \Gamma\Big(\sum_I c_I f_I\Big) \le \sum_I c_I \Gamma(f_I).$$
In view of the above expression of $\Gamma$, this amounts to show that for all $x\in \Delta_{n-1}$,
$$ \sum_{1\le i\neq j\le n}x_ix_j \Big(\sum_I c_I \partial_i f_I -\sum_I c_I \partial_j f_I\Big)^2
 \le \sum_I c_I  \sum_{1\le i\neq j\le n} x_ix_j \Big(\partial_i f_I-\partial_j f_I\Big)^2.$$
Hence it is sufficient to show that for all $i\neq j$, it holds
 $$  \Big(\sum_I c_I \big(\partial_i f_I - \partial_j f_I\big) \Big)^2 \le  \sum_I c_I \Big(\partial_i f_I-\partial_j f_I\Big)^2.$$
If $f_I(x)$ only depends on $(x_k)_{k\in I}$ then $ \partial_i f_I - \partial_j f_I=0$ if $\{i,j\}\cap I=\emptyset$.
Moreover if $f_I(x)=g(\sum_{k\in I}x_k)$ then   $ \partial_i f_I - \partial_j f_I=0$ also if $\{i,j\} \subset I$.
Hence the summations on $I$ actually only involve the sets $I\in \mathcal I_1$ such that $\{i,j\}\cap I\neq \emptyset$
and the sets $I\in \mathcal I_2$ such that $\mathrm{card}(\{i,j\}\cap I)=1$. By hypothesis the sum of the corresponding coefficients
$c_I$ is at most one, so that the required inequality is a mere consequence of the convexity of the square function.
Hence Condition~\eqref{cond:Gamma} holds true and we get the local inequality. By ergodicity the inequality passes
to the measure $\widetilde D_{n-1}( \alpha)$.
\end{proof}

Let $p>0$. Let $B_p^n=\{x\in\R^n;\; \sum_i |x_i|^p\le 1\}$ be the unit ball for the $\ell_p$ norm on $\R^n$.
On the corresponding unit sphere  $\partial B_p^n=\{x\in\R^n;\; \sum_i |x_i|^p=1 \}$, one often considers
the cone measure $\mu_p^n$ defined by $\mu_p^n(A)= \mathrm{Vol}_n([0,1].A)/ \mathrm{Vol}_n(B_p^n)$, $A\subset \partial B_p^n$.
Here $[0,1]\cdot A$ is the intersection of $B_p^n$ with the cone of apex at the origin spanned by $A$.

\begin{corollary}
   Let $X$ be a random vector on $\R^n$. Assume that it is either uniformly distributed on  $B_p^n$ or
   distributed according to the cone measure on $\partial B_p^n$. Then for all even functions $f_i:[-1,1]\to\R^+$
  $$ E\Big(\prod_{i=1}^{n} f_i(X_i)\Big) \le \prod_{i=1}^{n} E\Big(f_i(X_i)^2\Big)^{\frac12}.$$
\end{corollary}
\begin{proof}
   This is deduced from a particular case of the previous result on Dirichlet distributions, which ensures
that for $Y$ distributed according to $\widetilde D_{n-1}(\alpha)$, and $g_i:[0,1]\to \R^+,$
a similar inequality holds: $E\prod g_i(Y_i) \le \prod \big(Eg_i^2(Y_i)\big)^{1/2}$.
Indeed the uniform measure on $B_p^n$ and the cone measure on $\partial B_p^n$
can be viewed as symmetrized versions of the images of Dirichlet laws by maps of the 
form $T(x_1,\ldots,x_n)=(T_1(x_1),\ldots, T_n(x_n))$. Hence if we choose $g_i=f_i\circ T_i$ in the latter
inequality, we get the claim.
  Let us make this strategy explicit in the case of the cone measure. Let $\varepsilon_i$, $G_i$, $i=1,\ldots,n,$ 
be independent random variables. Assume that $\varepsilon_1$ is uniform on $\{-1,1\}$ and $G_i$ distributed
according to $e^{-t^p} dt/\Gamma(1+1/p)$. Then it is known that the vector 
$$ X=\frac{(\varepsilon_1G_1,\ldots,\varepsilon_nG_n)}{\Big(G_1^p+\cdots+G_n^p\Big)^{\frac{1}{p}}}$$
is distributed according to the cone measure.
Hence $|X_i|^p=G_i^p/(G_1^p+\cdots+G_n^p)$ where  $G_i^p$ is $Gamma(1/p)$-distributed. So applying
the Brascamp-Lieb inequality for $f_i(x)=g_i(x_i^{1/p})$ yields the claim.

A similar approach is possible for the uniform distribution on $B_p^n$ thanks to the representation
provided in \cite{bartgmn04pagl}.
\end{proof}

\begin{remark}
   The cone measure on $\partial B_2^n$ is simply the uniform measure on $\Sp^{n-1}$, for which a similar inequality
  holds for general functions $f_i$ (i.e. it is not necessary to assume that they are even). 
  Hence one way ask whether   the
 symmetry assumption in the previous corollary is really needed. In order to remove it one would 
  need a result for symmetrized   Dirichlet laws, namely for measures on $\partial B_1^n$ with density
with respect to Lebesgue  measure proportional to $\prod_i |x_i|^{\alpha_i-1}$. 
At first sight, there does not seem to 
be any problem to extend our approach. However the ergodicity of these measures is a delicate issue.
Indeed the fact that the density vanishes inside the domain may, in terms of the corresponding random process, create
potential barriers that may not be crossed or potential wells into which the process may get stuck. On the technical 
level, the domain of the operator may be too small to contain enough non-symmetric functions.
 \end{remark}

 \begin{remark}\label{rem:general}
   Proposition \ref{prop:Dirichlet} and  many results of this work, involve two kinds of functions which depend only on
   some coordinates $(x_k)_{k\in I}$ (some depend on all these coordinates and some depend on them only through their
   sum). It is possible to consider more general dependences. We have not tried to reach the highest generality in
    this respect. Let us briefly mention a quite general extension of Proposition~\ref{prop:Dirichlet}: 
    we could consider functions $f_I$ where $I=(I_1,\ldots, I_K)$ is a collection of disjoint subsets of $\{1,\ldots,n\}$, such that $f_I(x)$ only depends on 
    $$T_I(x):=\left(\sum_{i\in I_1}x_i,\ldots,\sum_{i\in I_K} x_i\right).$$
    One can check that the map $T_I$ commutes with the Fleming-Viot operator (this uses the disjointness of 
     $I_1,\ldots, I_K$). If one considers now a collection of functions $(f_I)_{I\in \mathcal I}$ and corresponding
      coefficients $( f_I)_{I\in \mathcal I}$, then a Brascamp-Lieb inequality holds provided for all $i\neq j$ in
       $\{1,\ldots,n\}$, $\sum_{I\in A_{i,j}} c_I\le 1$, where
       $$ A_{i,j}=\Big\{I\in \mathcal I;\;  \exists \ell,\, \mathrm{card}(I_\ell \cap \{i,j\})=1\Big\}.$$
       The proof follows the same arguments as the one of  Proposition \ref{prop:Dirichlet}. We omit the details. Note that several results of this 
paper can be extended in an analogous way.
\end{remark}

\section{Discrete models}\label{sec:disc}

In this section, we deal with discrete models, and in particular we have to use  the \BL-condition  in its brute form~\eqref{ineq:L} since we are no longer working with diffusion generators. We nevertheless
provide a simple criterion which can be worked out for a number of discrete models
of interest.

\subsection{Abstract criterion}
Throughout this paragraph, $E$ will thus be a finite or countable state space.
Let $K$ be a Markov kernel on $E$, that is, $K : E \times E \to [0, \infty)$ is such that for every $x \in E$,
$\sum_{y \in E} K(x,y) = 1$. If $f : E \to \rr$ is bounded, set
$Kf(x) = \sum_{y \in E} K(x,y) f(y)$, $x \in E$.
As before, for given maps $ T_i : E \to E_i$, $i = 1, \ldots, m$, we say they commute with $K$ if
for any function $f : E \to \rr $, $K(f \circ T_i)$ is a function of $T_i$.
Again, this amounts to the existence of a Markov kernel $K_i$
on $E_i$ such that $ K(f \circ T_i) = K_i(f) \circ T_i$. This definition is of course
equivalent to abstract one of \S\ref{sec:abs} in terms  of the associated Markov generator  
$$L = K - {\rm Id}.$$

The next proposition provides a simple equivalent
criterion for the \BL-condition~\eqref{ineq:L}  in this context.

\vskip 3mm

\begin{proposition}[\BL-condition in the discrete case]\label{prop:discret}  For distinct $x, y \in E $ such that $K(x,y) >0$, set
$$ I_{x,y} = \big\{ i \in\{1, \ldots , m\} ;\; T_i(x) \not= T_i(y) \big\} .$$
Let $c_i \geq 0$, $i = 1, \ldots, m $. Then the \BL-condition~\eqref{ineq:L}  holds if and only if
\begin{equation}\label{cond:discret}
 \sum_{i \in I_{x,y}} c_i  \leq 1  , \quad \textrm{ for all $ x \not= y$ in $E$ such that $K(x,y) > 0$}. 
\end{equation}
Therefore, under
this condition, for every non-negative functions $f_i : E_i \to \rr$, $i =1, \ldots, m$,
and every $t\geq 0$,
$$ P_t \bigg ( \prod_{i=1}^m f_i^{c_i} \circ T_i\bigg ) 
             \leq  \prod_{i=1}^m \big (P_t (f_i \circ T_i)\big ) ^{c_i} . $$

In particular, if $K$ has an ergodic invariant probability measure $\mu $ and if
%as $t \to \infty$,
%$$ \int \prod _{i=1}^m f_i ^{c_i} \circ T_i \, d\mu   
%         \leq  \prod _{i=1}^m \bigg ( \int f_i \circ T_i\,  d\mu \bigg )^{c_i} . $$
 for all $x, y \in E $ distinct with $K(x,y) >0$, it holds $ {\rm card } \,  \{ i = 1, \ldots , m ;\; T_i(x) \not= T_i(y) \}\le p $, then choosing
$c_i = \frac{1}{p}$, $i =1, \ldots, m$, we have that  
$$ \int \prod _{i=1}^m f_i  \circ T_i \, d\mu  \leq  \prod _{i=1}^m \left(\int  (f_i \circ T_i)^p d\mu \right)^{\frac1p}. $$
\end{proposition}

\begin{proof}  At fixed $ x \in E$, condition \eqref{ineq:L} may be written as
\begin{equation}\label{ineq:disc2} 
   \sum_{y \in E} K(x,y) \big (  e^{ \sum_{i=1}^m c_i [ f_i \circ T_i (y) - f_i \circ T_i(x)]} -1 \big)
     \leq \sum_{i=1}^m c_i \sum_{y \in E} K(x,y) \big (  e^{f_i \circ T_i (y) -  f_i \circ T_i(x)} - 1 \big). 
 \end{equation}
 The sums over $i$ on both sides only run over $i \in I_{x,y}$ so that the preceding inequality is equivalent to saying that
$$ \sum_{y \in E} K(x,y) \varphi\Big( \sum_{i \in I_{x,y}}c_i [ f_i \circ T_i (y) - f_i \circ T_i(x)]\Big)  
     \leq \sum_{y \in E} K(x,y) \sum_{i \in I_{x,y}} c_i \varphi\Big(f_i \circ T_i (y) -  f_i \circ T_i(x)\Big),$$
     where $\varphi(u)=e^u-1$. Since $\varphi(0)=0$, we can restrict the previous sum over $y\in E\setminus\{x\}$, and of course we can ask that $K(x,y)\neq 0$. Now, for fixed $x,y\in E$ with $x\neq y$ and $K(x,y)\neq 0$, we argue  that the Condition \eqref{cond:discret} on the
$c_i's$ combines with the convexity of $\varphi$ to give (pointwise) the desired inequality.

 Conversely,
if \eqref{ineq:disc2} holds for all choices of $f_i$, $i=1, \ldots,m,$ we choose  $f_i(z)=\theta \1_{z\neq T_i(x)}$
where $ \theta \in \mathbb R^+$. Letting $\theta\to +\infty$ and comparing the orders of the terms in \eqref{ineq:disc2} shows that for each $y\neq x$ with $K(x,y)\neq >0$, we must have $\sum_{i \in I_{x,y}} c_i \leq 1$.
\end{proof}

\begin{remark}[Extension to non-finite settings]
The careful reader has probably noticed that the finiteness (or countability) of $E$ is not central in the argument. All the argument works as soon as we can express $L+I=:K$ in terms of a Markov kernel. Indeed, this allows us to reduce the problem to a pointwise inequality.
\end{remark}

We next illustrate instances of the preceding result. 

%\vskip 2mm
\subsection{Examples}

\subsubsection{Homomorphisms of finitely generated groups}
Let for example $G$, $G_i$, $i=1, \ldots , m$, be finite or countable groups and
$T_i : G \to G_i$ be homomorphisms. Let $K$ be a Markov kernel on $G$.
It is clear that each $T_i$ commutes with $K$.

Assume furthermore that $K$ is left-invariant in the sense that 
$K(gx,gy) = K(x,y)$ for all $x,y,g \in G$.
We may let for example $G$ be finitely generated with generating set $S$,
and $K(x,y) = {\rm Card } \, (S)^{-1} {\rm 1}_S(y^{-1}x)$, $x, y \in G$.
Then, condition \eqref{cond:discret} of Proposition~\ref{prop:discret} amounts to
$$ \sum_{i \in I_z} c_i \leq 1 $$
for every $z \in S$ where $I_z = \{ i= 1, \ldots, m; z \notin {\rm Ker} \, (T_i) \}$.

\subsubsection{Coordinates of the symmetric group}  
Let $ E$ be the symmetric group $\mathcal S_n$ over $n$ elements $\{1, \ldots , n\}$,
$n \geq 2$. This set is the discrete analogue of $SO(n)$. Unlike the continuous setting, 
there are several possible choices for the kernel $K$. However in view of the latter proposition,
where each couple $(x,y)$ with $K(x,y)>0$ leads to a linear constraint on the exponents $c_i$, it is 
natural to take a small (or even minimal) generating set $S$ and to consider:
  $$K(x,y)=\frac{1}{\mathrm{card}(S)} \mbox{ if there is } \tau \in S \mbox{ with } y=\tau x.$$
 We choose for $S$ the set of all transpositions.
 %(this is not a minimal set, but
 %choosing a minimal generating subset would not improve our bounds).
  The following calculation will 
 show that it is  the best choice, since it minimizes the size of the support $\mathrm{supp}(\tau)=\{j;\;
 \tau(j)\neq j\}$. 
  
The normalized counting measure $\mu$ is invariant for  $K$. Actually $S$ being stable by inverse it is also reversible:
$$\int (Kf)\,  g\, d\mu=\int \frac{1}{\mathrm{card}(S)} \sum_{\tau\in S} f(\tau x)g(x) \, d\mu(x)
=  \int \frac{1}{\mathrm{card}(S)} \sum_{\tau\in S} f(y)g(\tau^{-1}y) \, d\mu(y)=\int (Kg)\, f \, d\mu.$$

 Let $I$ be a subset of $\{1, \ldots, n\}$. We consider  the map $T_I$
  defined  by
$$T_I(x) = x_{|I}=(x(i))_{i \in I}, \quad \forall x\in \mathcal S_n .$$
Then $T_I$ commutes with $K$; indeed
$$K(f\circ T_I)(x)=\frac{2}{n(n-1)} \sum_{\tau\in S} (f\circ T_I)(\tau x)$$
and $T_I(\tau x)=(\tau \circ x)_{|I}=\tau \circ x_{|I}$ depends only on $T_I(x)$.
The result of Proposition~\ref{prop:discret} involves the condition $T_I(x)\neq T_I(y)$ for $K(x,y)>0$.
Let us formulate it in a more concrete manner:
$$T_I(x)\neq T_I(\tau x) \Longleftrightarrow  \exists\,  i\in I,\; x(i)\neq \tau x(i) 
\Longleftrightarrow I \cap x^{-1} (\mathrm{supp}(\tau)) \neq \emptyset.$$
Note that since the proposition involves this condition for all $x\in \mathcal S_n$,
 the set $x^{-1}(\mathrm{supp}(\tau))$
can be any set with the size of the support of $\tau$. Choosing transpositions then clearly appears as the 
most economical choice.  

For  $I\subset \{1, \ldots, n\}$   we may also consider   the map $R_I$  defined by
$$R_I(x) = x(I)=\{x(i),\; i \in I\}, \quad \forall x\in \mathcal S_n .$$
Then $R_I$ also commutes with $K$ and for any $x$ and any transposition $\tau$,
$R_I(x)\neq R_I(\tau x)$ happens if and only if $\tau$ moves one point in $x(I)$ outside $x(I)$. Hence 
 $$R_I(x)\neq R_I(\tau x) \Longleftrightarrow  \mathrm{card}\big (I \cap x^{-1} (\mathrm{supp}(\tau))\big)=1.$$
  Combining these observations with Proposition~\ref{prop:discret} yields a discrete analogue to Proposition~\ref{prop:coordSOn}:
\begin{proposition} \label{prop:coordSn} Let $\mathcal I$ be a collection of subsets of  $\{1,\ldots,n\}$. Assume that it is written
as a disjoint union $\mathcal I=\mathcal I_1\cup \mathcal I_2$. 
For each nonempty subset $I\in \mathcal I$,  let  $c_I\ge 0$ 
and $f_I:\mathcal S_n\to \R^+$ such that 
  \begin{itemize}
        \item if $I\in \mathcal I_1$ then  for all $x$, $f_I(x)$ only depends  on $x_{|I}$, 
         \item if $I\in \mathcal I_2$ then for all $x$, $f_I(x)$ only depends on $x(I)$.
\end{itemize}
If for all $1\le i, j\le n$ with $i\neq j$ it holds:
  $$\displaystyle  \sum_{\stackrel{I\in \mathcal I_1}{I\cap\{i,j\}\neq\emptyset}} c_I
    +  \sum_{\stackrel{I\in \mathcal I_2}{ \mathrm{card}(I\cap\{i,j\})=1}} \!\!\!c_I\le 1,$$
  then the \BL-condition\eqref{ineq:L} is satisfied and 
  $$ \int_{\mathcal S_n} \prod_{I\in\mathcal I}  f_I^{c_I} d\mu \le \prod_{I\in\mathcal I}  \left(\int_{\mathcal S_n} f_I\, d\mu\right)^{c_I}.$$ 
\end{proposition}  
  The examples given after Proposition~\ref{prop:coordSOn} transfer to $\mathcal S_n$. For a family $\mathcal I$  of subsets of  $\{1,\ldots,n\}$, introduce the exponents:
  $$p = \max_{i\neq j} \mathrm{card}\big( \{ I \in \mathcal I  ; \ i\in I, \textrm{ or } j \in I\} \big)  \quad \textrm{and}\quad q= \max_{i\neq j} \mathrm{card}\big( \{ I \in \mathcal I  ; \ \mathrm{card}(I\cap\{i,j\})=1\} \big)  .$$
  Then, for functions $g_I$ and $h_I$ defined on suitable sets, we have 
  \begin{eqnarray*}
\int_{\mathcal S_n} \prod_{I\in \mathcal I} g_I(\sigma_{|I}) \, d\mu(\sigma)&\le& \prod_{I\in\mathcal I} \left(\int_{\mathcal S_n}  g_I(\sigma_{|I}) ^p \, d\mu(\sigma)
\right)^{\frac1p},\\
\int_{\mathcal S_n} \prod_{I\in \mathcal I} h_I(\sigma(I) ) \, d\mu(\sigma)&\le &\prod_{I\in\mathcal I} \left(\int_{\mathcal S_n}  h_I(\sigma(I))^q \, d\mu(\sigma)
\right)^{\frac1q}.
\end{eqnarray*}
A particular case of interest (where these two cases coincide) is when $\mathcal I =\big\{ \{1\}, \ldots, \{n\} \big\}$. Then, $p=q=2$ and we recover  the inequality on permanents given in \cite{C-L-L2}.

\subsubsection{Slices of the discrete cube and multivariate hypergeometric distributions}
For $n \geq k\geq 0$, let
$$ \Omega _{n,k} = \big \{ x \in \{0,1\}^n; x_1+ \cdots + x_n = k \big \}$$
equipped with uniform measure. These sets are discrete analogues of the sphere $S^{n-1}$.
Two elements $x,y$ in $\Omega _{n,k}$ are neighbors if and only   if they 
differ on exactly two coordinates, a relation written as $x\sim y$. Let $K$ be the nearest neighbor random
walk on $\Omega _{n,k}$ (known as the Bernoulli-Laplace model) defined by
$$ Kf (x) = \frac{1}{ k(n-k)} \sum_{y \sim x} f(y) . $$
It is easy to check that $Kf (x)$ only depends on the $i$'th coordinate $x_i$ of $x$ if
this is the case for $f$. Indeed, the number of neighbors $y$ of $x$
such that $y_i = x_i$ is equal to $(k-x_i) (n-1-k +x_i)$, whereas when
$y_i = 1- x_i$, this number is equal to the number of coordinates
$x_j$, $j \not= i$, such that $x_j = 1 -x _i$. For the coordinate maps
$T_i(x) = x_i$, $1 \leq i\leq n$, we are thus in the preceding setting of
commuting operators so that Proposition~\ref{prop:discret}  applies with $p=2$.

Alternatively one can use the following observation, which was pointed out to us by  P. Caputo. The uniform probability measure on $\Omega_{n,k}$ is the image of the uniform probability
measure on the permutation group $\mathcal S_n$ by the map $ x \in \mathcal S_n \mapsto ({\bf 1}_{x (i) \leq k})_{1\leq i\leq n}$. Consequently  the correlation inequalities derived on $\mathcal S_n$ for functions
depending on blocks of coordinates pass to $\Omega_{n,k}$ to yield the same result.
Such a reasoning may be extended in order to encompass more general distributions.
Consider integer numbers $K\le M$  and  $m=(m_i)_{1\le i\le n}$ such that
$\sum_i m_i=M$. The multivariate hypergeometric distribution $\mathcal H(m,K)$
is defined on $\mathbb N^n$ by 
$$ \mathcal H(m,K)(\{(k_1,\ldots,k_n)\})=\frac{\prod_{i=1}^m {m_i \choose k_i}}{{M \choose K}}$$
if $k_1+\cdots+k_n=K$ and for all $i$, $k_i\le m_i$ and $ H(m,K)(\{(k_1,\ldots,k_n)\})=0$ otherwise.
Given an urn containing $M$ balls of $n$ different colors, and more precisely $m_i$ of the $i^{\mathrm{th}}$ color,
if one draws $K$ balls (uniformly) at random then the $n$-tuple $(X_1,\ldots,X_n)$ consisting of the numbers 
of balls of each color in the sample is $\mathcal H(m,K)$-distributed.
It is not hard to check that $\mathcal H(m,K)$ coincides with the image of the uniform probability law on
the permutation group $\mathcal S_M$ by the map
$$ \sigma \in\mathcal S_M \mapsto T(\sigma):=\left(\mathrm{card}\Big\{j \in \Big[1+ \sum_{\ell\le i-1}m_\ell,\sum_{\ell\le i} m_\ell\Big];\;  \sigma(j)\le K\Big\}\right)_{i=1}^n.$$
This observation can be used to show that Proposition~\ref{prop:Dirichlet} remains valid if one
replaces  the Dirichlet laws by   multivariate 
hypergeometric distributions. We only outline the proof. 
Starting from functions $f_I$ defined on the 
support of $\mathcal H(m,K)$, we consider the functions $g_I:=f_I\circ T$.   Note that $g_I(\sigma)$ depends on the images by $\sigma$ of several intervals  of $\{1,\ldots,M\}$. Applying Proposition~\ref{prop:coordSn}
directly would not give the right result, since it only deals with simpler forms of dependences.
 Hence we  need to go back to Proposition~\ref{prop:discret}, in the spirit of the proof of Proposition~\ref{prop:coordSn} (this is actually
related to Remark~\ref{rem:general}). We omit the details.

\subsubsection{Product spaces and Finner's theorem} \label{sssec:finner}
Let us go back to more general distributions (including continuous distributions on non-finite spaces) but in the context of product structures.
The hypotheses in  Propositions~\ref{prop:coordSOn}, \ref{prop:coordSn} or \ref{prop:discret} are reminiscent of Finner's theorem
 \cite{F} which expresses that if $E = X_1 \times \cdots \times X_n$
is a product space with product probability measure $\mu = \nu _1 \otimes \cdots \otimes \nu _n$,
and if,   for $i = 1, \ldots, m$, $T_i : E \to E_i$ is the coordinate  projection on the space
$E_i := \prod_{j \in S_i} X_j $ determined by $S_i \subset \{1, \ldots , n\}$, then for any
non-negative functions $f_i : E_i \to \rr$, $i = 1, \ldots, m$,
$$ \int \prod _{i=1}^m f_i ^{c_i} \circ T_i \, d\mu   
         \leq  \prod _{i=1}^m \bigg ( \int f_i \circ T_i d\mu \bigg )^{c_i} $$
provided that
$$ \sum_{i;  S_i \ni j} c_i \leq  1 \quad {\hbox {for every}} \; \;  j = 1, \ldots, n. $$
This statement is actually contained in Proposition~\ref{prop:discret} for a suitable choice of
the kernel $K$. Without loss of generality, we may assume that, for each $i$, $X_i $
is a finite set equipped with a probability measure $\nu_i $ that charges all points.
Consider the kernels $K_i$ on $X_i$ given by $K_i(x_i,y_i) = \nu_i (y_i)$, and
tensorize them to the product space $E = X_1 \times \cdots \times X_n$ by
$$ {\cal K} = \frac{1}{n} \sum_{i=1}^n \tilde{I} \otimes \cdots \otimes \tilde{I}
             \otimes K_i \otimes \tilde{I} \otimes \cdots \otimes \tilde{I} $$
where $\tilde{I}$ is defined on $E_j$ by $\tilde{I}(x_j,y_j) = 1_{x_j =  y_j}$ (in other words, the associated Markov operator is the identity).
The commutation property of the projection operators $T_i$
is obvious. Moreover, for distinct elements $x,y $ in $E$,
${\cal K} (x,y) >0$ if and only if $x = (x_1, \ldots, x_n) $ and
$y = (y_1, \ldots, y_n) $ differ at exactly one coordinate, say $j$. Now
the set of $i$'s such that $T_i(x) \not= T_i(y)$ is exactly the set
of $i$'s such that $S_i \ni j$.

In particular, the preceding kernel provides a proof of the classical
H\"older inequality on the finite space $X$ equipped with the
probability measure $\nu $, and by approximation on any finite measure space.

\section{Sums of squares}\label{sec:square}

In this short paragraph, we briefly illustrate how the ideas developed in the preceding
discrete setting may also be of interest for classes of diffusion generators. Assume the generator $L$ is a sum 
of squares of vector fields  on a manifold $E$, 
$$L = \sum_{\ell } X_\ell^2.$$
 Let for example
$T_i : E \to \rr^{k_i}$, $i=1, \ldots, m$, be  commuting (with $L$) maps. We interpret $X_\ell T_i$ coordinate by coordinate.
The criterion put forward in Proposition~\ref{prop:discret} then
adapts to this setting:
\begin{proposition} For every $\ell$, let
$ I_\ell :=  \big\{ i\in \{1, \ldots , m\} ; X_\ell T_i \not= 0  \big\} $.
Let $c_i \geq 0$, $i = 1, \ldots , m$, be such that
$$ \sum_{i \in I_\ell } c_i \leq  1 \quad {\hbox {for every}} \; \; \ell . $$
Then, for every non-negative functions $f_i : E_i \to \rr$, $i =1, \ldots, m$,
and every $t \geq 0$,
$$ P_t \bigg ( \prod _{i=1}^m f_i^{c_i} \circ T_i \bigg )
                  \leq  \prod _{i=1}^m \big( P_t ( f_i \circ T_i  )\big) ^{c_i} . $$
                 In particular, if for all $\ell$, $  {\rm card } \, \{ i = 1, \ldots , m ; X_\ell T_i\not= 0 \}\le p$, we may choose $c_i = \frac{1}{p}$, $i =1, \ldots, m$.
\end{proposition} 
\begin{proof}
Since $\Gamma(f)=\sum_\ell (X_\ell f)^2$, according to Fact~\ref{BLdiffusion}, the  
 \BL-condition
\eqref{ineq:L} takes the form 
\begin{equation}\label{cond:squares}
  \sum_{\ell} (X_\ell H)^2 \le \sum_{i=1}^{m} c_i \sum_\ell \big(X_\ell (f_i\circ T_i) \big)^2,
 %\sum_{i=1}^m c_i  \sum_{\ell} \big [ X_\ell H - X_\ell (f_i \circ T_i) \big ] 
 %                    X_\ell (f_i \circ T_i)  \leq 0.
                \end{equation}     
where we recall that $H=\sum_{i=1}^{m} c_i f_i\circ T_i$.
Hence we are done if we can prove that for every $\ell$,
\begin{equation}\label{eq:claim-carre}
 \Big(\sum_{i=1}^{m} c_i  X_\ell (f\circ T_i) \Big)^2 \le \sum_{i=1}^{m} c_i \big( X_\ell (f\circ T_i) \big)^2.
\end{equation}
If $f_i$ is a function on $\rr^{k_i}$, then
$X_\ell (f_i\circ T_i) = \langle X_\ell T_i, \nabla f _i (T_i) \rangle$ is zero when $i\not\in I_\ell$.
Hence the summations in the above inequality only hold on $i\in I_\ell$.
Since, by hypothesis $\sum_{i\in I_\ell} c_i\le 1$, Inequality~\eqref{eq:claim-carre} is valid by convexity of the 
square function. The conclusion follows.
\end{proof}
We illustrate this result in the context of the Loomis-Whitney inequalities on the sphere. Consider
$$ \Delta =  \frac12 \sum_{k,\ell} X_{k\ell }^2 = 
          \frac12 \sum_{k,\ell} [ x_k \partial _\ell - x_\ell \partial _k ]^2 $$
the Laplace operator on the sphere $S^{n-1} \subset \rr^n$.
Let $A$ be a subset of $\{1, \ldots, n\}$ with $d$ elements,
and consider $T : \rr^n \to \rr^d$ defined by $T(x) = (x_i)_{i \in A}$. Then
$X_{k\ell } T_A = 0$ if and only if $\{k,\ell \} \cap A = \emptyset$. Thus, for every $k, \ell$, 
\begin{eqnarray*}  p   &=& {\rm card} \,  \big \{ A, |A| = d ; X_{k,\ell} T_A\not= 0  \big \} \\
          &=&  {n \choose d} - {n-2 \choose d} \\
          & =& {{n-1} \choose {d-1}} + {{n-2} \choose {d-1}} . 
\end{eqnarray*}
One instance of application is $d=1$ (for which $p=2$) from which we recover inequality \eqref{eq:BLsphere}
involving functions of $T_i(x)=x_i$. The approach here is indeed very close to the one of Carlen, Lieb and Loss~\cite{C-L-L1}.

\begin{remark}
   This viewpoint best explains the analogy between the results on $SO(n)$ and $\mathcal S_n$. Indeed the 
  infinitesimal rotation $ x_k \partial _\ell - x_\ell \partial _k$ in $\mathrm{vect}(e_k,e_\ell)$ is the 
  analogue of the transposition $\tau_{k,\ell}$.
\end{remark}

\section{Entropy of marginals}

In this section, we investigate, from the abstract Markov operator point of view, descriptions of the Brascamp-Lieb inequalities and entropy inequalities for marginals following \cite{C-L-L1,C-CE}. As in Section~\ref{sec:abs}, we do not make precise the classes of functions under consideration.

Let $(E,\mu)$ be a probability space and $T_i : E \to E_i$ be measurable maps.                          
Given a probability density $f$ on $E$ with respect
to $\mu $, denote by $f_i$ its conditional expectation with
respect to $T_i$. In other words, $f_i$ is the unique probability density on $E$ with respect to $\mu$ such that, for every bounded measurable
$\varphi : E_i \to \rr$,
\begin{equation}\label{eq:conditional}
 \int f \, \varphi \circ T_i \, d\mu  = \int f_i  \, \varphi \circ T_i \, d\mu .  
 \end{equation}
(Since $f_i = h_i \circ T_i$ for some $h_i : E_i \to \rr$, $h_i$ may be thought of as the ``marginal" of
$f$ in the direction of $T_i$.)
As shown in \cite{C-L-L1}, the Brascamp-Lieb inequality \eqref{ineq:mu} may be used, by standard arguments, to prove
the entropy inequality for the probability density $f$
\begin{equation}\label{ineq:ent}
 \sum_{i=1}^m c_i \int \! f_i  \log f_i  d\mu
       \leq  \int \! f \log f d\mu . 
       \end{equation} 
A recent work by    Carlen
and   Cordero-Erausquin \cite{C-CE}  shows  that there is a full equivalence: 

\begin{proposition} The following are equivalent.

(i) For every non-negative functions $g_i : E_i \to \rr$, $i = 1, \ldots , m$,
$$ \int  \prod_{i=1}^m g_i^{c_i} \circ T_i d \mu  
                \leq  \prod_{i=1}^m  \bigg ( \int g_i \circ T_i d\mu  \bigg ) ^{c_i}. $$
                
(ii) For every probability density $f$ with respect to $\mu $,
$$ \int f \log f d\mu  \geq \sum_{i=1}^m c_i \int f_i \log f_i d\mu .$$
\end{proposition}

Since semigroup proofs are available for Brascamp-Lieb inequalities, it is natural to 
hope for semigroup proofs of entropy inequalities. Such an approach was suggested in
  \cite{B-CE-M} for spherical measures,
   on the basis of the corresponding inequality for the Fisher information. 
  In the remainder of this section, we discuss the extension of this argument to the abstract 
  framework.
  
  Let $L$ be   a Markov generator on $E$ with semigroup $\Pt$. We require that $L$ be invariant, symmetric and ergodic for $\mu$. Denote by $\Gamma $  the carr\'e du champ operator of $L$ as defined in~\eqref{eq:carreduchamp}. Hence, the Dirichlet form is expressed as follows
  $$ \mathcal E(f,g)=\int \Gamma(f,g)\, d\mu=-\int f\, Lg\, d\mu=- \int g\, Lf\, d\mu.$$
  It is classical that, under suitable domain assumptions,
\begin{equation}\label{eq:ent-semigroup}  
 \int \! f \log f d\mu = \int_0^\infty \! dt \int  \Gamma (P_t f,\log P_t f)  \, d\mu.
 \end{equation}
 The Fisher information of a function $f>0$ is defined by
    $$J(f):=\mathcal E(f,\log f).$$
 The  above equality \eqref{eq:ent-semigroup} becomes 
 $$ \int \! f \log f d\mu = \int_0^\infty  J(P_t f)\,  dt $$
 and so, in view of the  commutation between $T_i$ and $G$, which ensures that 
 $$P_t (f_i) = (P_t f)_i,$$
we see hat the entropy inequality \eqref{ineq:ent} may be derived from its analogue for the Fisher information,

   The next result shows that such inequality for Fisher information can indeed be derived directly  from the \BL-condition in our  abstract setting. In view of the previous discussion, this therefore provides a different route for proving Brascamp-Lieb inequalities. 

 \begin{theorem}[Superadditivity of Fisher information]\label{theo:subinfo}
   Assume that $L$  is a Markov generator on $E$ which commutes with the maps $T_i  $  and that the \BL-condition~\eqref{ineq:L}  holds. Then,
for every probability density $f$ on $E$ with respect to $\mu $, under the preceding
notation,
\begin{equation}\label{eq:subinfo}
 \sum_{i=1}^m c_i \, J(f_i) \leq  J(f). 
 \end{equation}
\end{theorem}

Before proving this result in full generality, let us note that in the case  where  $L$ is a diffusion,   this theorem can be derived easily, following ideas from~\cite{B-CE-M}. Indeed, when $L$ is a diffusion we have  
$$J(f)=\int \frac{\Gamma(f)}{f}\, d\mu .$$
Using the definition of the conditional density~\eqref{eq:conditional} and the chain rule formula for $L$ 
we see that,  for each $i\le m$, 
$$J(f_i) =- \int f_i L(\log f_i) \, d\mu = -\int f L(\log f_i)  \, d\mu =  \int  \frac{\Gamma (f,  f_i)}{f_i} \,  d\mu .$$
Using the Cauchy-Schwarz inequality and~\eqref{eq:conditional} again we get
$$ J(f_i)^2
       \leq \int  \frac{\Gamma (f, f_i)^2}{f \, \Gamma (f_i)} \,  d\mu 
                   \int  \frac{ \Gamma (f_i) f}{ {f_i}^2 }  d\mu  
    = \int  \frac{\Gamma (f, f_i)^2}{f \, \Gamma (f_i)} \,  d\mu 
                   \int  \frac{ \Gamma (f,f_i)}{ f_i }  d\mu , 
$$
which means that 
$$ J(f_i)
     \leq  \int \frac{\Gamma ( f,f_i)^2}{ f \, \Gamma (f_i) } \, d\mu .  $$
We conclude to~\eqref{eq:subinfo} after noticing that  condition~\eqref{ineq:L} can be expressed in dual form as 
$$ \sum_{i=1}^m c_i \, \frac{ \Gamma (f,f_i)^2}{\Gamma (f_i) } \leq \Gamma (f). $$

Similar strategy however does not work in the non-diffusion  case. We present below a new method that allows us to treat the general case of a Markov generator. It relies on the following observation which is of independent interest.

\begin{lemma}\label{lem:fisher}
Assume $L$ is a  Markov generator  symmetric for $\mu$. Then for functions $f>0$ and $H$ of arbitrary sign on $E$ we have 
\begin{equation}\label{InfofH}
\mathcal E (f, H) \le  \mathcal E(f,\log f)+\int f e^{-H}L \big(e^H\big)\, d\mu.
\end{equation}
\end{lemma}
In other words we have the following dual formulation of Fisher information:
$$J(f) = \sup_H \left\{\mathcal E (f, H) - \int f e^{-H}L \big(e^H\big)\, d\mu \right\}.$$

\begin{proof}
We introduce the operator $P:=L+  {\rm Id}$ which is, as $L$, symmetric on $L^2(\mu)$. 
Replacing $L$ by $P - {\rm Id}$ we see that the inequality~\eqref{InfofH} to be proven rewrites as
\begin{equation}\label{InfofH2}
 \int f [H - PH]  d\mu  \leq  \int f\log f d\mu  - \int (Pf) \log f d\mu  - \int f d\mu
          + \int f e^{-H} P(e^H) d\mu .
 \end{equation}
By symmetry, the left-hand side is equal to $\int [P(fH) - HPf] d\mu $. By Young's inequality
$ab \leq a \log a - a + e^b$, $a>0$, $b \in \R$, we get that for every $\lambda >0$,
$$  P(fH) = \lambda P \Big ( \frac{f}{\lambda } \, H \Big)
        \leq   P(f \log f) - (Pf) \log \lambda  - Pf + \lambda P(e^H).$$
Hence, choosing $\lambda = f e^{-H}$,
$$ P(fH) - H Pf \leq P(f\log f) - (Pf) \log f - Pf + f e^{-H} P(e^H).$$
The desired inequality~\eqref{InfofH2} follows after integration, since for every $g$ we have $\int P g\, d\mu = \int g\, d\mu$. 
 \end{proof}

With the previous lemma in hand, we can easily complete the proof of the theorem.
\begin{proof}[Proof of Theorem~\ref{theo:subinfo}]
 Note that  the conditional expectation property yields, for every $i=1, \ldots, m$,
\begin{equation}\label{eq:Jmarg}
 J(f_i) = {\cal E} (f_i , \log f_i) =-\int f_i \,L(\log f_i) \,d\mu= -\int f\,  L(\log f_i) \,d\mu
 ={\cal E} (f, \log f_i).  
\end{equation}
Hence
$$ \sum_{i=1}^m c_i J(f_i) = \sum_{i=1}^m c_i \, {\cal E} (f, \log f_i)
   = {\cal E} (f, H),$$
where $ H = \sum_{i=1}^m c_i \log f_i$. Combining  Lemma~\ref{lem:fisher} and \BL-condition~\eqref{ineq:L}
(written for $F_i=\log f_i$ which is a function of $T_i$) we get
\begin{eqnarray*}
{\cal E} (f, H) &\le & \mathcal E(f,\log f)+\int f e^{-H}L \big(e^H\big)\, d\mu \\
   &\le & J(f)+ \int f \sum_i c_i  \frac{1}{f_i} \,  L(f_i) \, d\mu \\
  % &=& J(f)+ \sum_i c_i \int f  \frac{1}{f_i} \,  L(f_i) \, d\mu \\
   &=& J(f)+ \sum_i c_i \int   L(f_i) \, d\mu =J(f),
\end{eqnarray*}
where we have used in the last step
 that $L(f_i)/f_i$ is a function of $T_i$ and the conditional expectation property~\eqref{eq:conditional}.
\end{proof}

 Superadditive inequalities for Fisher information were considered on the sphere $S^{n-1}\subset\R^n$ in~\cite{B-CE-M} in the case of $T_i=P_{E_i}$ with the $E_i$ for subspaces $E_i\subset \R^n$ satisfying $\sum_i c_i P_{E_i}\le\mathrm{Id}_{\R^n}$. As explained in \S\ref{quotients},  the \BL-condition~\eqref{ineq:L} is verified for $d_i=c_i/2$ and we recover by the previous proposition the inequality from~\cite{B-CE-M}.

In the discrete case, some examples of superadditive inequalities for Fisher information were implicitly obtained in the papers~\cite{B, G-Q, G}. The goal of these papers is to prove modified log-Sobolev inequalities of the form
$$ \forall f:E\to \R^+\textrm{ with }  \int f \, d\mu = 1, \quad \rho_0 \int f\log f \, d\mu \le \mathcal E (f, \log f).$$
As pointed out to us by Eric Carlen, one can extract from their proofs (which is by induction)  superadditive
 inequalities for Fisher information which constitute a central technical ingredient. The main examples considered
 in theses papers are the symmetric group and slices of the discrete cube. There, the marginals are considered with
 respect to maps $T_i$ which belong to the family studied in the previous section, for which we have proved that 
 \BL-condition~\eqref{ineq:L}  holds, and for which we therefore have the desired superadditive inequalities.

\bigskip
\bigskip

\noindent
F. B., M. L.: Institut de Mathématiques de Toulouse (CNRS UMR 5219), Université Paul Sabatier, 31062 Toulouse cedex 9, France.  
barthe@math.univ-toulouse.fr, ledoux@math.univ-toulouse.fr
\smallskip

\noindent
D. C.-E.: 
Institut de Mathématiques de Jussieu (CNRS UMR 7586), Équipe d'Analyse Fonctionnelle,
Université Pierre et Marie Curie,
4, place Jussieu,
75252 Paris Cedex 05,
France.
cordero@math.jussieu.fr

\smallskip

\noindent
B. M.:  Laboratoire d'Analyse et de Mathématiques Appliquées (CNRS UMR 8050), 
Université de Marne-la-Vallée, 
77454 Marne-la-Vallée cedex 2, France.
bernard.maurey@univ-mlv.fr


\begin{thebibliography}{1}



\bibitem{bakr94huts}
D.~Bakry.
\newblock L'hypercontractivit\'e et son utilisation en th\'eorie des
  semigroupes.
\newblock In {\em Lectures on probability theory (Saint-Flour, 1992)}, volume
  1581 of {\em Lecture Notes in Math.}, pages 1--114. Springer, Berlin, 1994.

\bibitem{ball89vscr}
K.~M. Ball.
\newblock Volumes of sections of cubes and related problems.
\newblock In J.~Lindenstrauss and V.~D. Milman, editors, {\em Israel seminar on
  Geometric Aspects of Functional Analysis}, number 1376 in Lectures Notes in
  Math. Springer-Verlag, 1989.

\bibitem{bart98rfbl}
F.~Barthe.
\newblock On a reverse form of the {B}rascamp-{L}ieb inequality.
\newblock {\em Invent. Math.}, 134: 335--361, 1998.

\bibitem{B-CE}{F. Barthe, D. Cordero-Erausquin.} Inverse Brascamp-Lieb inequalities along the heat equation.
Geometric Aspects of Functional Analysis. Lecture Notes in Math. 1850, 65--71 (2004). Springer.

\bibitem{B-CE-M}{F. Barthe, D. Cordero-Erausquin, B. Maurey.}  Entropy of spherical marginals
and related inequalities. \emph{J. Math. Pures Appl.},  86: 89--99 (2006).

\bibitem{bartgmn04pagl}
F.~Barthe, O.~Gu{\'e}don, S.~Mendelson, and A.~Naor.
\newblock A probabilistic approach to the geometry of the $\ell_p^n$ ball.
\newblock {\em Ann. Probab.}, 33(2): 480--513, 2005.

\bibitem{B-C-C-T1}
{J. Bennett,  A.  Carbery,   M. Christ,  T.  Tao.}  The Brascamp-Lieb inequalities: finiteness, structure and
 extremals.
\emph{Geom. Funct. Anal.},  17(5): 1343--1415 (2008).

\bibitem{brasl76bcyi}
H.~J. Brascamp and E.~H. Lieb.
\newblock Best constants in {Y}oung's inequality, its converse and its
  generalization to more than three functions.
\newblock {\em Adv. Math.}, 20: 151--173, 1976.

\bibitem{B}
S. Bobkov, P. Tetali, Modified logarithmic Sobolev inequalities in discrete settings, \emph{J. Theoret. Probab.}, 19(2): 289--336 (2006). 

\bibitem{C-CE}{E. Carlen, D. Cordero-Erausquin.} 
Subadditivity of the entropy and its relation to Brascamp-Lieb type inequalities, preprint (2007), to appear  in \emph{ Geom. Funct. Anal.}


\bibitem{C-L-L1}{E. Carlen, E. Lieb, M. Loss.} A sharp analog of Young's inequality on $S^N$
and related entropy inequalities. \emph{J. Geom. Anal.}, 14: 487--520 (2004).

\bibitem{C-L-L2}{E. Carlen, E. Lieb, M. Loss.} An inequality of Hadamard type for permanents.
\emph{Methods Appl. Anal.}, 13: 1--17 (2006).

\bibitem{DS}  P. Diaconis, L.  Saloff-Coste.  Logarithmic Sobolev inequalities for finite Markov chains.
 \emph{Ann. Appl. Probab.}  6 (3): 695--750 (1996).

\bibitem{E-K}{S. Ethier, T. Kurtz.} Fleming-Viot processes in population genetics. \emph{SIAM J. Control Optim.},  31:  345--386 (1993).
 
\bibitem{F}{H. Finner.}  {A generalization of {H}\"older's inequality and some
              probability inequalities.} \emph{Ann. Probab.}, 20(4): 1893--1901 (1992).
    
\bibitem{G-Q}  F. Gao,  J. Quastel. Exponential decay of entropy in the random transposition and Bernoulli-Laplace models, \emph{Ann. Appl. Probab.}, 13(4): 1591--1600 (2003).

\bibitem{G} S. Goel.
Modified logarithmic Sobolev inequalities for some models of random walk, 
\emph{Stochastic Process. Appl.}, 114(1): 51--79 (2004). 

\bibitem{lieb90gkgm}
E.~H. Lieb.
\newblock Gaussian kernels have only Gaussian maximizers.
\newblock {\em Invent. Math.}, 102: 179--208, (1990).


\bibitem{M-R}{Z. Ma, M. Röckner.} Introduction to the theory of (non-symmetric) Dirichlet forms. Springer, Berlin (1992).
 
\bibitem{V}
 {S. Valdimarsson.} Optimizers for the Brascamp-Lieb inequality.
 Israel J. Math.,  168:  253--274 (2008).

\end{thebibliography}
\end{document}